\documentclass[11pt]{article}

\usepackage[draft]{ed}
\usepackage{amsfonts}
\usepackage{amsmath}
\usepackage[linesnumbered,lined,boxed,commentsnumbered]{algorithm2e}
\usepackage{amssymb}
\DeclareMathOperator{\csch}{csch}
\DeclareMathOperator{\shch}{shch}
\usepackage{epsfig}
\usepackage{bm}
\usepackage{enumerate}
\numberwithin{equation}{section} 

\usepackage{color}
\usepackage[normalem]{ulem}

\title{Sweeping preconditioners for the iterative solution of quasiperiodic Helmholtz transmission problems in layered media}
\author{David Nicholls, Carlos P\'erez-Arancibia, Catalin Turc}

\newtheorem{theorem}{Theorem}[section]
\newtheorem{lemma}[theorem]{Lemma}

\newtheorem{remark}[theorem]{Remark}
\newenvironment{proof}{\hspace{0.5cm} {\bf Proof.}}
{$\quad {}_\blacksquare$\vspace{0.3cm}}

\date{}
\newcommand{\triple}[1]{{\left\vert\kern-0.25ex\left\vert\kern-0.25ex\left\vert #1 
    \right\vert\kern-0.25ex\right\vert\kern-0.25ex\right\vert}}

\begin{document}
\maketitle
\begin{abstract}
  We present a sweeping preconditioner for quasi-optimal Domain Decomposition Methods (DDM) applied to Helmholtz transmission problems in periodic layered media. Quasi-optimal DD (QO DD) for Helmholtz equations rely on transmission operators that are approximations of Dirichlet-to-Neumann (DtN) operators. Employing shape perturbation series, we construct approximations of DtN operators corresponding to periodic domains, which we then use as transmission operators in a non-overlapping DD framework. The Robin-to-Robin (RtR) operators that are the building blocks of  DDM are expressed via robust boundary integral equation formulations. We use Nystr\"om discretizations of quasi-periodic boundary integral operators to construct high-order approximations of RtR. Based on the premise that the quasi-optimal transmission operators should act like perfect transparent boundary conditions, we construct an approximate LU factorization of the tridiagonal QO DD matrix associated with periodic layered media, which is then used as a double sweep preconditioner. We present a variety of numerical results that showcase the effectiveness of the sweeping preconditioners applied to QO DD for the iterative solution of Helmholtz transmission problems in periodic layered media.
   
\end{abstract}

\section{Introduction}
\label{intro}

The numerical simulation of interactions between electromagnetic, acoustic, and elastic waves with periodic layered media has numerous applications in the fields of optics, photonics, geophysics~\cite{Petit}. Given the important technological applications of periodic layered media, the simulation of wave propagation in such environments has attracted significant attention~\cite{cho2015robust,LaiKobayashiBarnett2015,nicholls2018stable,schadle2007domain,perez2018domain}. Regardless of the type of discretization (finite elements, finite differences, boundary integral operators), iterative solvers are the preferred method of solution especially for high-frequency layered configurations that involve large numbers of layers that may contain inclusions. The iterative solution  of high-frequency Helmholtz and Maxwell equations in complex media is a challenging computational problem~\cite{gander2016class}.  A successful strategy to tackle this problem relies on sweeping preconditioners~\cite{engquist2011sweeping}. We present in this paper a sweeping preconditioner for a Domain Decomposition formulation of Helmholtz transmission problems in two dimensional periodic layered  media.

Domain Decomposition Methods are natural candidates for the solution of Helmholtz transmission problems in periodic layered media~\cite{schadle2007domain,perez2018domain,nicholls2018stable}. Local subdomain solutions (the subdomains may or may not coincide with the periodic layers) are linked iteratively via Robin type transmission conditions defined on inter domain interfaces. Ideally, the transmission operators should act as transparent boundary conditions that allow information to flow out of each subdomain with very little information being reflected back. As such, for a given subdomain, optimal transmission operators on the subdomain interface consist of Dirichlet-to-Neumann (DtN) operators associated with the adjacent subdomain that shares the same interface. In practice, the transmission operators are constructed via various approximations of DtN operators that rely either on Fourier calculus~\cite{boubendirDDM,Gander1} or Perfectly Matched Layers~\cite{zepeda2016method,vion2014double}; the ensuing DD are referred to as Quasi-Optimal DD (QO DD) or Optimized Schwartz Methods~\cite{gander2016class}. 

The main scope of this paper is QO DD for the solution of Helmholtz transmission problems in periodic layered media separated by grating profiles (i.e. graphs of periodic functions). We present two strategies of subdomain partition whereby (1) the subdomains coincide with the layer subdomains and the subdomain interfaces coincide with the grating profiles of material discontinuity of the layered medium; and (2) the subdomains consist of horizontal strips whose flat boundaries do not intersect any of the grating profiles of material discontinuity. We note that the DD partition strategy (2) is only applicable to layered media configurations where the width of the layers is larger than the roughness of their interfaces. In each subdomain a local Helmholtz quasi-periodic equation with generalized Robin conditions must be solved (the wavenumber may be discontinuous in case (2)), and generalized Robin data on the subdomain boundaries are linked with those corresponding to the adjacent subdomain. The generalized Robin data corresponding to a given subdomain is defined in terms of transmission operators that are approximations of DtN operators corresponding to the adjacent subdomain. Such approximations of periodic DtN operators can be obtained via high-order shape perturbation/deformation series in case (1)~\cite{nicholls2004shape}. Specifically, using as a small parameter the roughness/elevation height of the grating, the periodic DtN operators are expressed as a perturbation series whose terms can be computed recursively. The zeroth order terms of the perturbation series coincide with DtN of layered domains with flat interfaces, which can be written explicitly in terms of Fourier multipliers. In the case of the subdomain partition (2), since the subdomain interfaces are flat,  the transmission operators are chosen to be the aforementioned Fourier multipliers. We establish that the ensuing QO DD corresponding to both subdomain partitions are equivalent to the original transmission problem, with the caveat that the roughness of the grating profiles must be small enough for the subdomain partition (1).

The exchange of Robin data amongst the subdomains in DD is realized via quasi-periodic Robin-to-Robin (RtR) maps that map incoming to outgoing  subdomain Robin data. Following the methodology introduced in~\cite{perez2018domain}, we express quasi-periodic RtR operators in terms of robust boundary integral equation formulations. The discretization of the RtR maps is realized by extending the high-order Nystr\"om method based on trigonometric interpolation and windowing quasi-periodic Green functions~\cite{perez2018domain} to the case of DtN transmission operators. Since the terms in the shape deformation series expansions of DtN operators are expressed in terms of Fourier multipliers~\cite{nicholls2004shape}, the discretization of the QO transmission operators is straightforward within the framework of trigonometric interpolation. Using Nystr\"om discretization RtR matrices, we discretize the QO DD formulation for layered transmission problems in the form a block tridiagonal matrix which we invert using Krylov subspace iterative methods. However, the numbers of iterations required for the solution of QO DD linear systems grows with the number of layers, especially for high frequencies at high-contrast layer media configurations. In order to alleviate this situation, we construct a double sweep preconditioner based on an approximate LU factorization of the block tridiagonal QO DD matrix. As it was very nicely explained in the recent contribution~\cite{gander2016class}, all of the effective sweeping preconditioners that have been introduced in the last decade~\cite{zepeda2016method,vion2014double,engquist2011sweeping} can be elegantly described in terms of optimized DD with layered subdomain partitions and approximate LU of the ensuing block tridiagonal DD matrices. The key insight in our construction of the LU factorization is related to the observation that if the transmission operators were to behave as perfect transparent boundary conditions, certain blocks in the QO DD matrix can be approximated by zero~\cite{vion2014double}. This approximation renders the LU factorization particularly simple as it bypasses altogether the need for inversions of block matrices.

We present a variety of numerical results that highlight the benefits of QO DD  formulations for the solution of transmission problems in periodic layered media, as well as the effectiveness of the sweeping preconditioners in the presence of large numbers of layers at high frequencies. With regards to the latter regime, we find that the sweeping preconditioners used in conjunction with QO DD with horizontal strip subdomain partitions is particularly effective. We mention that the quasi-optimal transmission operators based on Fourier square-root principal symbols approximations of DtN operators has been already used in several contributions~\cite{boubendirDDM,vion2014double,jerez2017multitrace}; we simply extend the square root Fourier calculus to the periodic setting and incorporate it within the high-order shape deformation expansions technology introduced in~~\cite{nicholls2004shape}. Furthermore, the construction of the sweeping preconditioners that we employ in this paper was originally introduced in~\cite{vion2014double} and further elaborated upon in~\cite{gander2016class}. The main contributions of this paper are (a) the integration of these two important ideas within a high-order Nystr\"om discretization of robust quasi-periodic boundary integral equation formulations of RtR maps, as well as (b) the analysis of the quasi-periodic QO DD. A fully parallel implementation of the sweeping preconditioner applied to QO DD for transmission Helmholtz equations in layered media is currently being developed. Also, the generalization of the DD with horizontal strip subdomain partitioning is currently under investigation; this would entail careful treatment of cross points (i.e. points on the subdomain boundaries where the wavenumbers are discontinuous), which we plan to pursue along the lines of the contribution~\cite{jerez2017multitrace}.

The paper is organized as follows: in Section~\ref{MS10} we present the formulation of Helmholtz transmission problems in periodic layered media; in Section~\ref{DDM} we present QO DD formulations of the periodic Helmholtz transmission problem; we continue in Section~\ref{trans_ops} with the construction of quasi-optimal transmission operators based on high-order shape perturbation series; we  show in Section~\ref{rtr} a means to express the QO DD RtR operators in terms of robust quasi-periodic boundary integral equation formulations, which, in turn, enable us to analyze the equivalence between the QO DD formulations and the original Helmholtz transmission problems; and we conclude in Section~\ref{num} with the construction of the sweeping preconditioner and with a presentation of a variety of numerical results that illustrate the effectiveness of these preconditioners in the context considered in this paper. 

\parskip 2pt plus2pt minus1pt

\section{Scalar transmission problems \label{MS10}}

We consider the problem of two dimensional quasi-periodic scattering by penetrable homogeneous periodic layers. We assume that the layers are given by $\Omega_j=\{(x_1,x_2)\in\mathbb{R}^2: \overline{F_j}+F_{j}(x_1)\leq x_2\leq \overline{F_{j-1}}+ F_{j-1}(x_1)\}$ for $0<j<N$ and $\Omega_0=\{(x_1,x_2)\in\mathbb{R}^2:\overline{F_0}+F_0(x_1)\leq x_2\}$ and $\Omega_{N+1}=\{(x_1,x_2)\in\mathbb{R}^2:x_2\leq \overline{F_N}+F_N(x_1)\}$, and all the functions $F_j$ are periodic with principal period $d$, that is $F_j(x_1+d)=F_j(x_1)$ for all $0\leq j\leq N$, and $\overline{F_j}\in\mathbb{R}, 0\leq j\leq N$. We assume that the medium occupying the layer $\Omega_j$ is homogeneous and its permitivity is $\epsilon_j$; the wavenumber $k_j$ in the layer $\Omega_j$ is given by $k_j=\omega\sqrt{\epsilon_j}$. We assume that a plane wave $u^{inc}(\mathbf{x})=\exp(i(\alpha x_1+i\beta x_2))$ where $\alpha^2+\beta^2=k_0^2$ impinges on the layered structure, and we are interested in looking for $\alpha$-quasi-periodic fields  $u_j$ that satisfy the following system of equations:
\begin{equation}\label{system_t}
\begin{array}{rclll}
 \Delta u_j +k_j^2 u_j &=&0& {\rm in}& \Omega_j^{per}:=\{(x_1,x_2)\in\Omega_j: 0\leq x_1\leq d\},\smallskip\\
  u_j+\delta_0u^{inc}&=&u_{j+1}&{\rm on}& \Gamma_j=\{(x_1,x_2):0\leq x_1\leq d,\ x_2=\overline{F_j}+F_j(x_1)\},\smallskip\\
  \gamma_j(\partial_{\nu_j}u_j+\delta_0\partial_{\nu_j}u^{inc})&=&-\gamma_{j+1}\partial_{\nu_{j+1}}u_{j+1}&{\rm on}& \Gamma_j,\smallskip
\end{array}\end{equation}
where $\nu_j$ denote the unit normals to the boundary $\partial\Omega_j$ pointing to the exterior of the subdomain $\Omega_j$ (i.e. for the domain $\Omega_0$ we define $\nu_0(x_1)=(F'_{0}(x_1),-1)/(1+(F'_{0}(x_1))^2)^{1/2}$ on $\Gamma_{0}$, for the domains $\Omega_j,\ 0<j<N$ we define $\nu_j(x_1)=(-F'_{j-1}(x_1),1)/(1+(F'_{j-1}(x_1))^2)^{1/2}$ on $\Gamma_{j-1}$ and $\nu_j(x_1)=(F'_{j}(x_1),-1)/(1+(F'_{j}(x_1))^2)^{1/2}$ on $\Gamma_{j}$, and finally, for the domain $\Omega_N$ we define $\nu_N(x_1)=(-F'_{N}(x_1),1)/(1+(F'_{N}(x_1))^2)^{1/2}$ on $\Gamma_{N}$) and $\delta_0$ is the Dirac distribution supported on $\Gamma_0$. We also denote by $n_0(x_1)=(F'_{0}(x_1),-1)$ on $\Gamma_{0}$ for the domain $\Omega_0$, for the domains $\Omega_j,\ 0<j<N$, $n_j(x_1)=(-F'_{j-1}(x_1),1)$ on $\Gamma_{j-1}$ and $n_j(x_1)=(F'_{j}(x_1),-1)$ on $\Gamma_{j}$, and finally, for the domain $\Omega_N$, $n_N(x_1)=(-F'_{N}(x_1),1)$ on $\Gamma_{N}$; we point out that $\nu_j=n_j/|n_j|$ for all domains $\Omega_j,\ 0\leq j\leq N$. We note that with this convention on unit normals we have that $\nu_j=-\nu_{j+1}$ as well as $n_j=-n_{j+1}$ on $\Gamma_j$. We also assume that $u_0$ and $u_{N}$ in equations~\eqref{system_t} are radiative in $\Omega_0$ and $\Omega_{N+1}$ respectively. The latter requirement amounts to expressing the solutions $u_0$ and $u_{N+1}$ in terms of Rayleigh series
\begin{equation}\label{eq:rad_up}
  u_0(x_1,x_2)=\sum_{r\in\mathbb{Z}} B_r^{+}e^{i\alpha_r x_1+i\beta_{0,r} x_2},\quad x_2>\overline{F_0}+\max{F_0}
\end{equation}
and
\begin{equation}\label{eq:rad_down}
  u_{N+1}(x_1,x_2)=\sum_{r\in\mathbb{Z}} B_r^{-}e^{i\alpha_r x_1-i\beta_{N+1,r} x_2},\quad x_2<\overline{F_N}+\min{F_N}
\end{equation}
where $\alpha_r=\alpha+\frac{2\pi}{d}r$, and $\beta_{0,r}=(k_0^2-\alpha_r^2)^{1/2}$, $\beta_{N+1,r}=(k_{N+1}^2-\alpha_r^2)^{1/2}$ where the branches of the square roots in the definition of $\beta_{0,r}$ and $\beta_{N+1,r}$ in such a way that $\sqrt{1}=1$, and that the branch cut coincides with the negative imaginary axis. We assume that the wavenumbers $k_j$ and the quantities $\gamma_j$ in the subdomains $\Omega_j$ are positive real numbers. In electromagnetic applications, $\gamma_j=1$ or $\gamma_j=\epsilon_j^{-1}$ depending whether the incident radiation is TE or TM. For the sake of simplicity,  we consider in this contribution the case $\gamma_j=1$; extensions to general positive $\gamma_j$ are straightforward.
\begin{figure}
\centering
\includegraphics[scale=1]{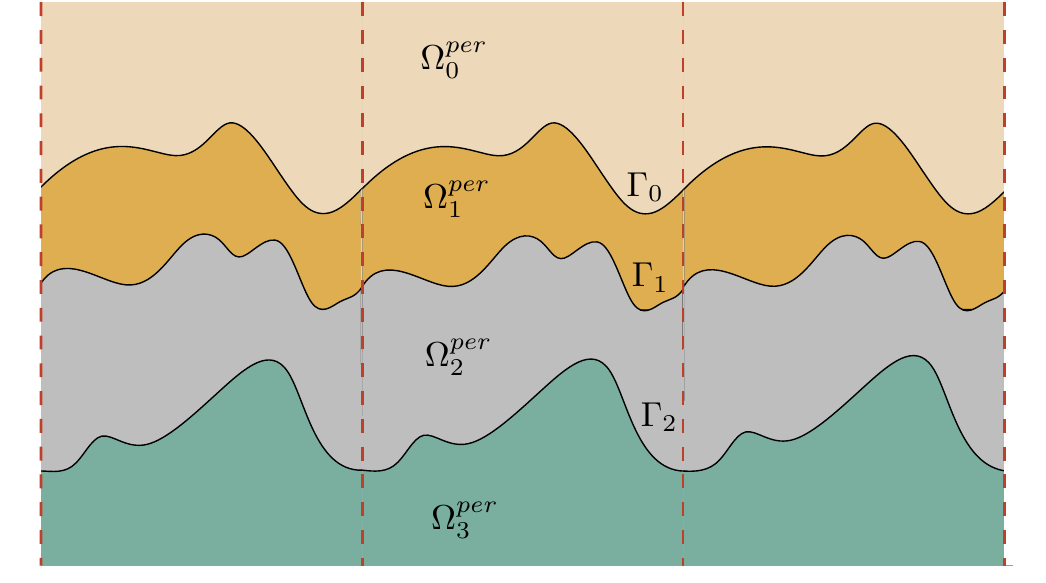}
\caption{Typical periodic layer structure with $N=3$.}
\label{fig:subdiv1}
\end{figure}

 \section{Domain decomposition approach\label{DDM}}

 The transmission problem~\eqref{system_t} can be formulated via boundary integral equations (BIEs)~\cite{arens2010scattering,cho2015robust} or via non-overlapping Domain Decomposition (DD)~\cite{perez2018domain,nicholls2018stable}.  Upon discretization, both the BIE and DD amount to solving block tridiagonal linear systems. In the case of large numbers of layers, the ensuing (large) linear systems are solved via direct methods~\cite{cho2015robust,perez2018domain} that rely on Schur complements. As such, the applicability of direct solvers for the numerical solution of the transmission problem~\eqref{system_t} is limited by the size of the Schur complements. Iterative solvers, on the other hand, do not suffer from the aforementioned size limitations, yet are challenged by the presence of significant multiple scattering, especially in high-contrast multi-layer configurations at high frequencies. In the high-frequency regime, relevant to technological applications, efficient preconditioners are needed in order to alleviate multiple scattering. The main scope of this contribution is to present such a preconditioner (referred to as the sweeping preconditioner~\cite{zepeda2016method,vion2014double,engquist2011sweeping}) in the context of DD formulation of quasiperiodic transmission problems.
 
 The main idea of DD is to divide the computational domain into subdomains, and to match subdomain quasiperiodic solutions of Helmholtz equations via Robin type transmission conditions on the subdomain interfaces. We consider in what follows two strategies of partitioning the computational domain into non-overlapping subdomains: the most natural one in which the DD subdomains coincide with the layer domains $\Omega_j^{per}$, and an alternative one in which the subdomains are horizontal strips. We present in what follows the details of the first subdomain partitioning strategy mentioned above.   

 \subsection{DD with subdomains $\Omega_j^{per}$}
 
 A natural non-overlapping domain decomposition approach for the solution of equations~\eqref{system_t} consists of solving subdomain problems in $\Omega_j^{per},j=0,\ldots,N+1$ with matching Robin transmission boundary conditions on the common subdomain interfaces $\Gamma_j$ for $j=0,\ldots,N$. Indeed, this procedure amounts to computing $\alpha$-quasiperiodic  subdomain solutions:
\begin{eqnarray}\label{DDM_t}
  \Delta u_j +k_j^2 u_j &=&0\qquad {\rm in}\quad \Omega_j^{per},\\
  (\partial_{n_0}u_0+\partial_{n_0}u^{inc})+Z_{1,0}(u_0+u^{inc})&=&-\partial_{n_{1}}u_{1}+Z_{1,0}u_{1}\ {\rm on}\ \Sigma_{0,1}:=\Gamma_{0}\nonumber\\
  \partial_{n_{1}}u_{1}+Z_{0,1}u_{1}&=&-(\partial_{n_0}u_0+\partial_{n_0}u^{inc})+Z_{0,1}(u_0+u^{inc})\ {\rm on}\ \Sigma_{1,0}:=\Gamma_{0}\nonumber\\
  \partial_{n_j}u_j+Z_{j+1,j}u_j&=&-\partial_{n_{j+1}}u_{j+1}+Z_{j+1,j}u_{j+1}\ {\rm on}\ \Sigma_{j,j+1}:=\Gamma_{j},\ 1\leq j\leq N\nonumber\\
  \partial_{n_{j+1}}u_{j+1}+Z_{j,j+1}u_{j+1}&=&-\partial_{n_j}u_j +Z_{j,j+1}u_j\ {\rm on}\ \Sigma_{j+1,j}:=\Gamma_{j}, 1\leq j\leq N.\nonumber
\end{eqnarray}
where $Z_{j+1,j}:H^{1/2}(\Sigma_{j,j+1})\to H^{-1/2}(\Sigma_{j,j+1})$, $Z_{j,j+1}:H^{1/2}(\Sigma_{j+1,j})\to H^{-1/2}(\Sigma_{j+1,j})$ are certain transmission operators for $0\leq j\leq N$, and $\partial_{n_j}$ denote normal derivatives with respect to the non-unit normals $n_j$. In addition, we require that $u_0$ and $u_{N+1}$ are radiative. We have chosen to double index the interfaces between layer subdomains: the first index $j$ refers to the index of the layer $\Omega_j$, whereas the second index $\ell$ denotes the index of the layer $\Omega_\ell$ adjacent to the layer $\Omega_j$ so that $\Sigma_{j,\ell}$ is the interface between $\Omega_j$ and $\Omega_\ell$. Here and in what follows $H^s(\Gamma)$ denote Sobolev spaces of $\alpha$-quasiperiodic functions/distributions defined on the periodic interface $\Gamma$; the definition of these spaces is given in terms of Fourier series~\cite{perez2018domain}.

Heuristically, in order to give rise to rapidly convergent iterative DD, the transmission operators $Z_{j+1,j}$ ought to be good approximations of the restriction to $\Sigma_{j+1,j}=\Sigma_{j,j+1}$ of the DtN operator associated with the $\alpha$-quasiperiodic Helmholtz equation in the domain $\Omega_{j+1}$ with wavenumber $k_{j+1}$. This requirement explains why the indices are reversed in the definition of the transmission operators. In addition, the transmission operators $Z_{j+1,j}$ and $Z_{j,j+1}$ ought to be selected to meet the following two criteria: (1) the subdomain boundary value problems that incorporate these transmission operators in the form of generalized Robin boundary conditions are well-posed for all frequencies, and (2) the DD matching of the generalized Robin data on the interfaces of material discontinuity (which coincide with the layer boundaries) is equivalent to the original transmission conditions~\eqref{system_t} on the same interfaces. 

Specifically, with regards to the issue (1) above, we require that for a given layer domain $\Omega_j$ with $1\leq j\leq N$, the following $\alpha$-quasiperiodic boundary value problem is well-posed:
\begin{eqnarray}\label{eq:H}
  \Delta w_j+k_j^2w_j&=&0\quad{\rm in}\ \Omega_j^{per}\\
  \partial_{n_j}w_j+Z_{j-1,j}w_j&=&g_{j,j-1}\quad{\rm on}\ \Sigma_{j,j-1}\nonumber\\
  \partial_{n_j}w_j+Z_{j+1,j}w_j&=&g_{j,j+1}\quad{\rm on}\ \Sigma_{j,j+1}\nonumber
  \end{eqnarray}
where $g_{j,j-1}$ and $g_{j,j+1}$ are generic $\alpha$-quasiperiodic functions defined on $\Sigma_{j,j-1}$ and $\Sigma_{j,j+1}$ respectively. The following coercivity properties
\begin{equation}\label{eq:well-pos_t}
  \Im \langle Z_{j-1,j}\varphi_{j,j-1},\ \varphi_{j,j-1}\rangle<0\quad\mbox{and}\quad \Im \langle Z_{j+1,j}\varphi_{j,j+1},\ \varphi_{j,j+1}\rangle<0,
\end{equation}
for all $\varphi_{j,j-1}\in H^{1/2}(\Sigma_{j,j-1}),\ \varphi_{j,j+1}\in H^{1/2}(\Sigma_{j,j+1})$ in terms of the $H^{1/2}$ and $H^{-1/2}$ duality pairings $\langle\cdot,\cdot\rangle$ are sufficient conditions for guaranteeing the well posedness of the boundary value problems~\eqref{eq:H}. Indeed, this can be established easily by an application of the Green's identities in the domain $\Omega_j^{per}$. In the case of the semi-infinite domain $\Omega_0$, we require that the following $\alpha$-quasiperiodic boundary value problem is well-posed:
\begin{eqnarray}\label{eq:H0}
  \Delta w_0+k_0^2w_0&=&0\quad{\rm in}\ \Omega_0^{per}\\
  \partial_{n_0}w_0+Z_{1,0}\ w_0 &=&g_{0,1}\quad{\rm on}\ \Sigma_{0,1}\nonumber
\end{eqnarray}
where $g_{0,1}$ is a $\alpha$-quasiperiodic function defined on $\Sigma_{0,1}$. The coercivity property
\begin{equation}\label{eq:well-pos_0}
  \Im \langle Z_{1,0}\varphi_{0,1},\ \varphi_{0,1}\rangle <0,\quad{\rm for\ all}\ \varphi_{0,1}\in H^{1/2}(\Sigma_{0,1}),
\end{equation}
suffices to establish the well posedness of the boundary value~\eqref{eq:H0}. The latter fact can be established via the same arguments as those in Theorem 3.1 in~\cite{perez2018domain}. A similar coercivity condition imposed on the operator $Z_{N,N+1}$ ensures the well posedness of the analogous $\alpha$-quasiperiodic boundary value problem on the semi-infinite domain $\Omega_{N+1}$.

Returning to the requirement (2) above, we ask that the DD matching of the generalized Robin data
\begin{eqnarray*}
 \partial_{n_j}u_j+Z_{j+1,j}u_j&=&-\partial_{n_{j+1}}u_{j+1}+Z_{j+1,j}u_{j+1}\quad{\rm on}\quad \Sigma_{j,j+1},\ 1\leq j\leq N\nonumber\\
  \partial_{n_{j+1}}u_{j+1}+Z_{j,j+1}u_{j+1}&=&-\partial_{n_j}u_j +Z_{j,j+1}u_j\quad{\rm on}\quad \Sigma_{j+1,j}, 1\leq j\leq N
\end{eqnarray*}
is equivalent to the continuity conditions
\[
u_j = u_{j+1}\qquad{\rm and}\qquad \partial_{n_j}u_j=-\partial_{n_{j+1}}u_{j+1}\quad{\rm on}\quad \Gamma_{j}=\Sigma_{j,j+1}=\Sigma_{j+1,j}, 1\leq j\leq N.
\]
It can be immediately seen that the equivalence in part (2) is guaranteed provided that $Z_{j,j+1}+Z_{j+1,j}:H^{1/2}(\Gamma_j)\to H^{-1/2}(\Gamma_j)$ is an injective operator. Under the assumption that the coercivity properties~\eqref{eq:well-pos_0} hold, it follows that
\[
\Im \langle(Z_{j,j+1}+Z_{j+1,j})\varphi,\ \varphi\rangle<0,\quad{\rm for\ all}\ \varphi\in H^{1/2}(\Gamma_j),
\]
and thus the operators $Z_{j,j+1}+Z_{j+1,j}$ are injective for all $1\leq j\leq N$. Thus, the coercivity properties~\eqref{eq:well-pos_0} ensure that both requirements (1) and (2) above are met. We postpone the discussion on the selection of the transmission operators $Z_{j,j+1}$ and $Z_{j+1,j}$ and we formulate the DD system~\eqref{DDM_t} in matrix operator form. To that end, we define certain RtR operators associated with the boundary value problems~\eqref{eq:H}.  Specifically, we define the RtR map $\mathcal{S}^j$ in the following manner:
\begin{equation}\label{RtRboxj_t}
   \mathcal{S}^j\begin{bmatrix}g_{j,j-1}\\g_{j,j+1}\end{bmatrix}=\begin{bmatrix}\mathcal{S}^j_{j-1,j-1} & \mathcal{S}^j_{j-1,j+1}\\ \mathcal{S}^j_{j+1,j-1} & \mathcal{S}^j_{j+1,j+1}\end{bmatrix}\begin{bmatrix}g_{j,j-1}\\g_{j,j+1}\end{bmatrix}:=\begin{bmatrix}(\partial_{n_j}w_j-Z_{j,j-1}w_j)|_{\Sigma_{j,j-1}}\\(\partial_{n_j}w_j-Z_{j,j+1}w_j)|_{\Sigma_{j,j+1}}\end{bmatrix}.
 \end{equation}
Also, associated with the boundary value problem~\eqref{eq:H0} posed in the semi-infinite domain $\Omega_0$ we define the RtR map $\mathcal{S}^0$ in the form
\begin{equation}\label{eq:S0}
  \mathcal{S}^0_{1,1}g_{0,1}:=(\partial_{n_0}w_0-Z_{0,1}w_0)|_{\Sigma_{0,1}}.
\end{equation}
The RtR map $\mathcal{S}^{N+1}_{N,N}$ corresponding to the domain $\Omega_{N+1}$ is defined in a similar manner to $\mathcal{S}^0_{1,1}$ but  for a boundary data $g_{N+1,N}$ defined on $\Sigma_{N+1,N}$. 

With these notations in place, the DD formulation~\eqref{DDM_t} seeks to find the generalized Robin data associated with each interface $\Gamma_j=\Sigma_{j,j+1}=\Sigma_{j+1,j}$
\begin{eqnarray*}
  f_{j}=\begin{bmatrix}f_{j,j+1}\\ f_{j+1,j}\end{bmatrix}&:=&\begin{bmatrix}(\partial_{n_j}u_j+Z_{j+1,j}u_j)|_{\Sigma_{j,j+1}}\\ (\partial_{n_{j+1}}u_{j+1}+Z_{j,j+1} u_{j+1})|_{\Sigma_{j+1,j}}\end{bmatrix},\ 0\leq j\leq N\\
\end{eqnarray*}
as the solution of the following $(2N+2)\times (2N+2)$ operator linear system
\begin{equation}\label{ddm_t_exp}
  \mathcal{A}f=b
\end{equation}
where $f=[f_0\ f_1\ \ldots f_{N}]^\top$ and the right-hand-side vector $b=[b_0\ b_1\ \ldots\ b_{N}]^\top$ has zero components $b_\ell=[0\ 0]^\top,\ 1\leq \ell\leq N$ with the exception of the first component 
 \begin{equation}\label{rhs_ddm_t}
   b_0=\begin{bmatrix}-(\partial_{n_0}u^{inc}+Z_{1,0}\ u^{inc})|_{\Sigma_{0,1}}\\-(\partial_{n_0}u^{inc}-Z_{0,1} u^{inc})|_{\Sigma_{1,0}}\end{bmatrix}\nonumber.
 \end{equation} 
and the DD matrix $\mathcal{A}$ is a tridiagonal block matrix given in explicit form by
\begin{equation}\label{eq:mA}
\mathcal{A}=\begin{bmatrix}
D_0 & U_0 & 0 &\ldots & 0\\
L_0 & D_1 & U_1 & \ldots & 0\\
\ldots & \ldots & \ldots & \ldots & \ldots\\
\ldots &L_{j-1} & D_j & U_j & \ldots\\
\ldots & \ldots & \ldots & \ldots & \ldots\\
\ldots & \ldots & L_{N-2} & D_{N-1} & U_{N-1}\\
\ldots & \ldots & \ldots & L_{N-1} & D_N\\
\end{bmatrix}
\end{equation}
where
\begin{equation}\label{eq:components}
D_j:=\begin{bmatrix} I & \mathcal{S}^{j+1}_{j,j} \\ \mathcal{S}^j_{j+1,j+1}& I \end{bmatrix}\quad U_j=\begin{bmatrix} \mathcal{S}^{j+1}_{j,j+2} & 0 \\ 0 & 0\end{bmatrix}\quad L_j=\begin{bmatrix}0 & 0 \\ 0 & \mathcal{S}^{j+1}_{j+2,j} \end{bmatrix}.
\end{equation}
We present in what follows a different strategy of domain decomposition whereby the subdomains are horizontal strips.

\subsection{DD with stripes subdomains}

An alternative DD possibility is to partition the computational domain using horizontal stripes. We restrict ourselves to cases where the layer domains $\Omega_j^{per},\ 1\leq j\leq N$ are wide enough so that each periodic interface $\Gamma_j,\ 0\leq j\leq N$ can be contained in a horizontal strip that does not intersect any other interface $\Gamma_\ell,\ \ell\neq j$. Under this assumption, these horizontal stripes constitute the DD subdomains---see Figure~\ref{fig:new_ddm} for a depiction of the partitioning in the case of three layers (i.e. $N=1$). In general, however, a domain decomposition into horizontal stripes might require that an interface $\Gamma_j$ intersect a (flat) boundary of a strip; we leave this challenging scenario for future considerations.

Assuming that there exist real numbers $c_0>c_1>\ldots >c_{N+1}$ such that for all $0\leq j\leq N$ we have that $c_j>\overline{F}_j+\max F_j(x_1)$ and $c_{j+1}< \overline{F}_j+\min F_j(x_1)$, then  we can partition $\mathbb{R}^2$ into a union of nonoverlapping horizontal strips $\mathbb{R}^2=\cup_{j=0}^{N+2} \Omega_j^\flat$, where the strip domains are defined as $\Omega_0^\flat:=\{(x_1,x_2): x_2\geq c_0\}$, $\Omega_j^\flat:=\{(x_1,x_2): c_j\leq x_2\leq c_{j-1}\},\ 1\leq j\leq N+1$, and $\Omega_{N+2}^\flat:=\{(x_1,x_2): x_2\leq c_{N+1}\}$. Using the domain decomposition into layered stripes we seek $\alpha$-quasiperiodic solutions $v_j$ of the following system of PDEs 
\begin{eqnarray}\label{DDM_t_final}
  \Delta v_j +k_j(x)^2 v_j &=&0\qquad {\rm in}\quad \Omega_j^{\flat,per},\\
  k_0(x):=k_0,\ k_j(x)&:=&\begin{cases}k_{j-1},\ x_2>\overline{F}_{j-1}+F_{j-1}(x_1)\\ k_j,\ x_2<\overline{F}_{j-1}+F_{j-1}(x_1)\end{cases},\ 1\leq j\leq N+1\nonumber\\
  \left[v_j\right]=0,\quad \left[\partial_{n_j}v_j\right]&=&0\quad{\rm on}\ \Gamma_{j-1}\nonumber\\
  -(\partial_{x_2}v_0+\partial_{x_2}u^{inc})+Z_{1,0}^\flat(v_0+u^{inc})&=&-\partial_{x_2}v_{1}+Z_{1,0}^\flat v_{1}\ {\rm on}\ \Sigma_{0,1}^\flat\nonumber\\
  \partial_{x_2}v_{1}+Z_{0,1}^\flat v_{1}&=&(\partial_{x_2}v_0+\partial_{x_2}u^{inc})+Z_{0,1}^\flat(v_0+u^{inc})\ {\rm on}\ \Sigma_{1,0}^\flat\nonumber\\
  -\partial_{x_2}v_j+Z_{j+1,j}^\flat v_j&=&-\partial_{x_2}v_{j+1}+Z_{j+1,j}^\flat v_{j+1}\ {\rm on}\ \Sigma^\flat_{j,j+1},\ 1\leq j\leq N+1\nonumber\\
  \partial_{x_2}v_{j+1}+Z_{j,j+1}^\flat v_{j+1}&=&\partial_{x_2}v_j +Z^\flat _{j,j+1}v_j\ {\rm on}\ \Sigma^\flat_{j+1,j}, 1\leq j\leq N+1,\nonumber
\end{eqnarray}
where $\Sigma^\flat_{j,j+1}=\Sigma_{j+1,j}^\flat:=\{(x_1,c_j), 0\leq x_1\leq d\},\ 0\leq j\leq N+1$ and $\left[v_j\right]$ denotes the jump of the function $v_j$ across the interface $\Gamma_{j-1}$. We require that the transmission operators have the following mapping properties $Z_{j+1,j}^\flat:H^{1/2}(\Sigma^\flat_{j,j+1})\to H^{-1/2}(\Sigma^\flat_{j,j+1})$ and $Z_{j+1,j}^\flat:H^{1/2}(\Sigma^\flat_{j+1,j})\to H^{-1/2}(\Sigma^\flat_{j,j+1})$ and satisfy coercivity properties similar to those in equations~\eqref{eq:well-pos_t}. 
\begin{figure}
\centering
\includegraphics[scale=1]{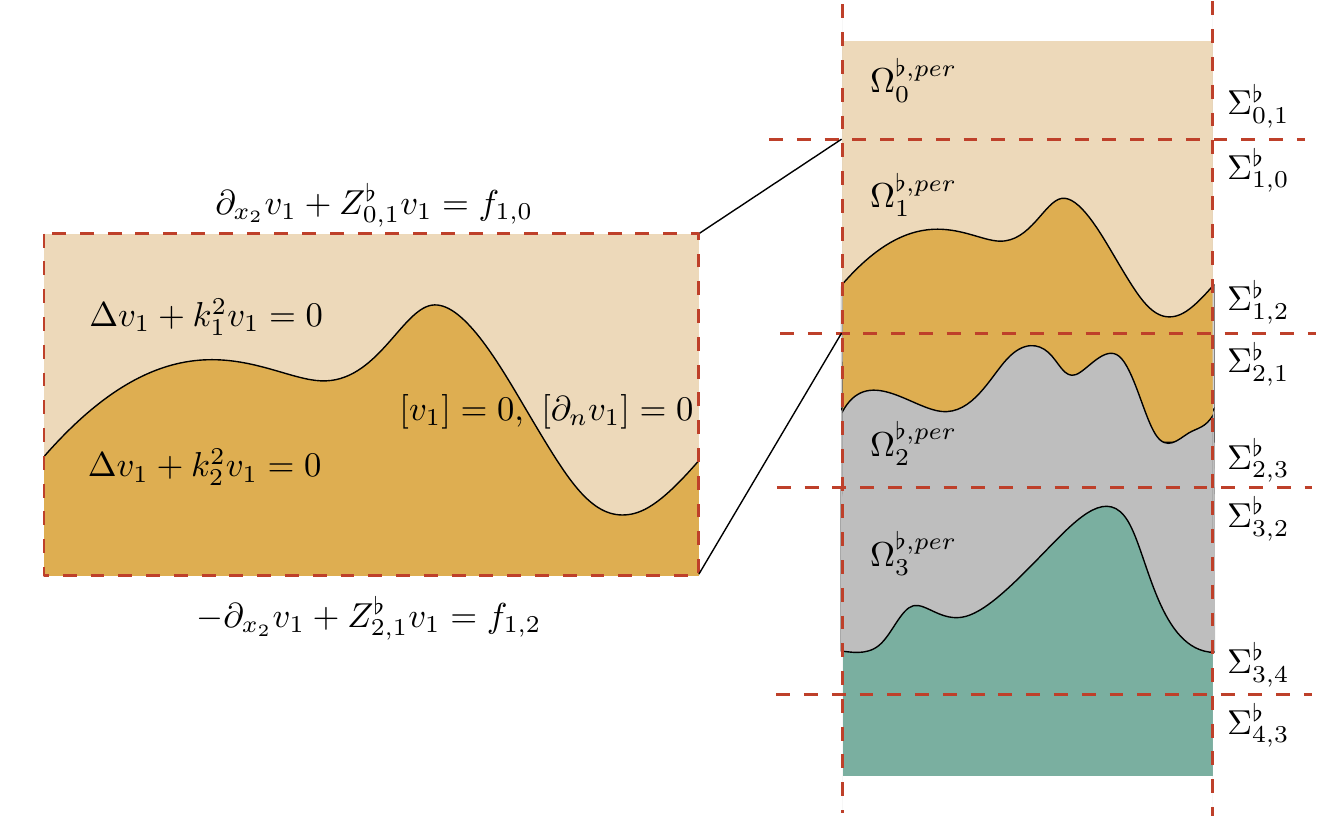}
\caption{New domain decomposition.}
\label{fig:new_ddm}
\end{figure}

The coercivity properties of the transmission operators $Z_{j-1,j}^\flat$ and $Z_{j+1,j}^\flat$ are needed to ensure the well-posedness of the  following subdomain PDEs
\begin{eqnarray}\label{DDM_t_final_well_p}
  \Delta v_j +k_j(x)^2 v_j &=&0\qquad {\rm in}\quad \Omega_j^{\flat,per},\\
  k_0(x):=k_0, k_j(x)&:=&\begin{cases}k_{j-1},\ x_2>\overline{F}_{j-1}+F_{j-1}(x_1)\\ k_j,\ x_2<\overline{F}_{j-1}+F_{j-1}(x_1)\end{cases},\nonumber\\
  \left[v_j\right]=0,\quad \left[\partial_{n_j}v_j\right]&=&0\quad{\rm on}\ \Gamma_{j-1}\nonumber\\
  \partial_{x_2}v_{j}+Z_{j-1,j}^\flat v_{j}&=&g^\flat_{j,j-1}\ {\rm on}\ \Sigma^\flat_{j,j-1},\nonumber\\
  -\partial_{x_2}v_{j}+Z_{j+1,j}^\flat v_{j}&=&g^\flat_{j,j+1}\ {\rm on}\ \Sigma^\flat_{j+1,j},\nonumber
\end{eqnarray}
for all $1\leq j\leq N+1$ as well as those posed in the semi-infinite domains $\Omega_0^{\flat,per}$ and $\Omega_{N+1}^{\flat,per}$ respectively. Associated to the Helmholtz transmission problem~\eqref{DDM_t_final_well_p} is the RtR operator defined below
\begin{equation}\label{RtRboxj_t_flat}
   \mathcal{S}^{\flat,j}\begin{bmatrix}g^\flat_{j,j-1}\\g^\flat_{j,j+1}\end{bmatrix}=\begin{bmatrix}\mathcal{S}^{\flat,j}_{j-1,j-1} & \mathcal{S}^{\flat,j}_{j-1,j+1}\\ \mathcal{S}^{\flat,j}_{j+1,j-1} & \mathcal{S}^{\flat,j}_{j+1,j+1}\end{bmatrix}\begin{bmatrix}g^\flat_{j,j-1}\\g^\flat_{j,j+1}\end{bmatrix}:=\begin{bmatrix}(\partial_{x_2}v_j-Z^\flat_{j,j-1}v_j)|_{\Sigma^\flat_{j,j-1}}\\(-\partial_{x_2}v_j-Z^\flat_{j,j+1}v_j)|_{\Sigma^\flat_{j,j+1}}\end{bmatrix}.
\end{equation}
The DD formulation~\eqref{DDM_t_final} then seeks to find the generalized Robin data associated with each interface $\Sigma^\flat_{j,j+1}=\Sigma^\flat_{j+1,j}$
\begin{eqnarray*}
  f_{j}^\flat=\begin{bmatrix}f^\flat_{j,j+1}\\ f^\flat_{j+1,j}\end{bmatrix}&:=&\begin{bmatrix}(-\partial_{x_2}v_j+Z^\flat_{j+1,j}v_j)|_{\Sigma_{j,j+1}}\\ (\partial_{x_2}v_{j+1}+Z^\flat_{j,j+1} v_{j+1})|_{\Sigma_{j+1,j}}\end{bmatrix},\ 0\leq j\leq N+1,\\
\end{eqnarray*}
as the solution of the following $(2N+4)\times (2N+4)$ operator linear system
\begin{equation}\label{ddm_t_exp_final}
  \mathcal{A}^\flat f^\flat=b^\flat
\end{equation}
where the DD matrix $\mathcal{A}^\flat$ is similar to that defined in equation~\eqref{eq:mA}, $f^\flat=[f_0^\flat\ f_1^\flat\ \ldots f_{N+1}^\flat]^\top$ and the right-hand-side vector $b^\flat=[b_0^\flat\ b_1^\flat\ \ldots\ b_{N+1}^\flat]^\top$ has zero components $b_\ell^\flat=[0\ 0]^\top,\ 1\leq \ell\leq N+1$ with the exception of the first component 
 \begin{equation}\label{rhs_ddm_t}
   b_0^\flat=\begin{bmatrix}(\partial_{x_2}u^{inc}-Z^\flat_{1,0}\ u^{inc})|_{\Sigma_{0,1}}\\(\partial_{x_2}u^{inc}+Z^\flat_{0,1} u^{inc})|_{\Sigma_{1,0}}\end{bmatrix}\nonumber.
 \end{equation} 
 Having described two possible DD strategies for the solution of quasi-periodic Helmholtz transmission problems~\eqref{system_t}, we present next a methodology based on Fourier calculus to construct quasi-optimal transmission operators.

\section{Construction of quasi-optimal transmission operators based on shape perturbation series\label{trans_ops}}

We present in what follows a perturbative method to construct quasi-optimal transmission operators $Z_{j,j+1}$ and $Z_{j+1,j}$ for $0\leq j\leq N$ corresponding to the DD formulation~\eqref{DDM_t}.  To this end, given a generic $d$-periodic profile function $F(x_1)$ we define the periodic interface $\Gamma:=\{(x_1,F(x_1)),\ 0\leq x_1\leq d\}$ and the semi-infinite domains $\Omega^{+,per}:=\{(x_1,x_2), 0\leq x_1\leq d,\ F(x_1)\leq x_2\}$ and respectively $\Omega^{-,per}:=\{(x_1,x_2), 0\leq x_1\leq d,\ F(x_1)\geq x_2\}$. We assume that the profile function $F(x_1)$ can be expressed in the form $F(x_1)=\varepsilon \widetilde{F}(x_1)$, where the $d$-periodic function $\widetilde{F}(x_1)$ is analytic (it actually suffices that the profile function is Lipschitz). We employ a perturbative approach~\cite{nicholls2004shape} to construct approximations of the DtN operator $Y^{\pm}(k,F)g:=\pm\partial_{n}v|_{\Gamma}$ corresponding to the following boundary value problem in the domains $\Omega^{\pm,per}$:
\begin{eqnarray}\label{eq:H00}
  \Delta v^{\pm}+k^2v^{\pm}&=&0\quad{\rm in}\ \Omega^{\pm,per}\\
   v^{\pm} &=&g\quad{\rm on}\ \Gamma\nonumber
\end{eqnarray}
where $v^{\pm}$ are radiative in the domains $\Omega^{\pm,per}$ and $g$ is a $\alpha$-quasiperiodic function defined on $\Gamma$, and $n(x)=(F'(x),-1)$ is the normal to $\Gamma$ pointing into the domain $\Omega^{-,per}$. Under the assumptions above, the DtN operators $Y^\pm(k,F)$ are analytic in the shape perturbation variable $\varepsilon$, and thus we seek the operator $Y^{\pm}(k,F)$ in the form of the perturbation series
\begin{equation}\label{eq:pert_series_0}
Y^{\pm}(k,F)=\sum_{n=0}^\infty Y_{n}^{\pm}(k,\widetilde{F})\varepsilon^n
\end{equation}
where the operators $Y_{n}^{\pm}(k,\widetilde{F}):H^{1/2}(\Gamma)\to H^{-1/2}(\Gamma)$ can be computed via explicit recursive formulas~\cite{nicholls2004shape}. Let us denote by $\rho(k,F)$ the radius of convergence of the perturbation series~\eqref{eq:pert_series_0}. Following~\cite{nicholls2004shape}, we present next the recursive formulas that lead to closed form expressions of the operators $Y_{n}^\pm(k,\widetilde{F})$. First, given an $\alpha$-quasiperiodic function $\varphi\in H^{1/2}(\Gamma)$ which can be represented as
\[
\varphi(x_1)=\sum_{n\in\mathbb{Z}}\varphi_{n}e^{i\alpha_n x_1}
\]
we define the Fourier multiplier operator
\begin{equation}\label{eq:Z_01}
  \beta_D(k)[\varphi](x_1):=\sum_{n\in\mathbb{Z}}\beta_{k,n}\varphi_{n}e^{i\alpha_n x_1},\ \beta_{k,n}:=(k^2-\alpha_n^2)^{1/2}.
\end{equation}
Then, it can be shown that the operators $Y_{n}^+(k,F)$ in the perturbation series~\eqref{eq:pert_series_0} can be computed via the recursion
\begin{eqnarray}\label{eq:recursive_1}
  Y_{0}^\pm(k,\widetilde{F})[\varphi]&=&(-i\beta_D(k))[\varphi]\nonumber\\
  Y_{n}^\pm(k,\widetilde{F})[\varphi_0]&=&\pm k^2\widetilde{F}_{n}(x_1)(\pm i\beta_D(k))^{n-1}\varphi\pm\partial_{x_1}\left[\widetilde{F}_{n}(x_1)\partial_{x_1}(\pm i\beta_D(k))^{n-1}\varphi\right]\nonumber\\
  &-&\sum_{m=0}^{n-1}Y_{m}^\pm(k,\widetilde{F})\left[\widetilde{F}_{n-m}(\pm i\beta_{D}(k))^{n-m}\varphi\right]
  \end{eqnarray}
where $\widetilde{F}_{\ell}(x_1):=\frac{\widetilde{F}(x_1)^\ell}{\ell !}$.  We note that $Y_{n}^\pm(k,\widetilde{F}):H^{1/2}(\Gamma)\to H^{-1/2}(\Gamma)$ for all  $0\leq n$, which is to say that none of the higher-order terms $Y_{n}^\pm(k,\widetilde{F}),\ 1\leq n$ is more regular that the lowest term $Y_{0}^\pm(k,\widetilde{F})$. Furthermore, since the recursions~\eqref{eq:recursive_1} may lead to significant subtractive cancellations, more stable expressions of the operators $Y_{n}^\pm(k,\widetilde{F}), 1\leq n\leq 2$ were proposed. Indeed, using the commutator
\[
\left[\beta_D(k),\widetilde{F}\right][\varphi]:=\beta_D(k)[\widetilde{F}\varphi]-\widetilde{F}\beta_D(k)[\varphi]
\]
it can be shown that the low-order term corrections $Y_{n}^\pm(k,\widetilde{F}),\ n=1,2$ can be expressed in the equivalent form
\begin{equation}\label{eq:T1}
  Y_{1}^\pm(k,\widetilde{F})[\varphi]=(D\widetilde{F})\ (D\varphi)-\left[\beta_D(k),\widetilde{F}\right][\beta_D(k)\varphi]
\end{equation}
and
\begin{equation}\label{eq:T2}
  Y_{2}^\pm(k,\widetilde{F})[\varphi]=i\beta_D(k)\left(-\left[\beta_D(k),\widetilde{F}^2/2\right][\beta_D(k)\varphi]+\widetilde{F}\left[\beta_D(k),\widetilde{F}\right][\beta_D(k)\varphi]\right)
\end{equation}
where $D=\partial_{x_1}$. However, the calculation of high-order correction terms $Y_{n}^+(k,\widetilde{F}),\ 3\leq n$ via the stable recursions above becomes quite cumbersome. As such, a different strategy based on changes of variables (that straighten out the boundary $\Gamma$) and DtN corresponding to variable coefficient Helmholtz equation in half-planes is advocated in~\cite{nicholls2004shape} for stable computations of DtN maps. Given that our motivation is to construct readily computable DD transmission operators that are approximations of DtN operators, we will restrict to low-order terms $Y_{n}^+(k,F)$ in the perturbation series~\eqref{eq:pert_series_0}, which, as discussed above, can be computed by explicit and stable recursions.

In order to meet the coercivity requirements~\eqref{eq:well-pos_t}, we {\em complexify} the wavenumber $k$ in the form $\kappa=k+i\sigma,\sigma>0$ and we define
\begin{equation}\label{eq:Z_01}
  Y^{\pm}_L(\kappa,F):=\sum_{\ell=0}^L Y_{\ell}^\pm(\kappa,\widetilde{F})\varepsilon^\ell,\ L\leq 2
\end{equation}
using formulas~\eqref{eq:T1} and~~\eqref{eq:T2} for the definition of the operators in equation~\eqref{eq:Z_01}. Indeed, we establish the following
\begin{lemma}\label{coerc1}
Provided that $\varepsilon<\rho(k,F)$ is small enough, the following coercivity property holds
\[
\Im\langle Y^\pm_L(\kappa,F)\varphi,\ \varphi\rangle<0
\]
for all $\varphi\in H^{1/2}(\Gamma)$.
\end{lemma}
\begin{proof}
By the construction of the Fourier multiplier operator $-i\beta_D(\kappa)$ we have that
\[
\Im\langle Y_{0}^\pm(\kappa,\widetilde{F})\varphi,\ \varphi\rangle=-\sum_{n\in\mathbb{Z}}\Re(\kappa^2-\alpha_n^2)^{1/2}|\varphi_{n}|^2<0
\]
for all $\varphi\in H^{1/2}(\Gamma_0)$, given that $\Re(\kappa^2-\alpha_n^2)^{1/2}>0$ for all $n\in\mathbb{Z}$. Using the fact that $\left|\langle Y_{\ell}^\pm(\kappa,\widetilde{F})\varphi,\ \varphi\rangle\right|\lesssim \|\varphi\|^2_{H^{1/2}(\Gamma)}$ we obtain
\[
\Im\langle  Y^\pm_L(\kappa,F)\varphi,\ \varphi\rangle\leq\Im\langle Y_{0}^\pm(\kappa,\widetilde{F})\varphi,\ \varphi\rangle+C\varepsilon \|\varphi\|^2_{H^{1/2}(\Gamma)}<0
\]
for $\varepsilon$ small enough.
\end{proof}

We are now in the position to construct quasi-optimal transmission operators $Z_{j-1,j}$ and $Z_{j+1,j}$. We assume without loss of generality that each grating profile $F_j(x_1)=\varepsilon \widetilde{F}_j(x_1),\ 0\leq j\leq N$, and we select transmission operators in the form
\begin{equation}\label{eq:layered_pm}
  Z_{j-1,j}^{s,L}:=Y_L^{+}(\kappa_{j-1},\widetilde{F}_{j-1}),\ 1\leq j\leq N+1,\qquad Z_{j+1,j}^{s,L}:=Y^{-}_L(\kappa_{j+1},\widetilde{F}_j),0\leq j\leq N,
\end{equation}
where $\kappa_j=k_j+i\sigma_j,\ \sigma_j>0$ and $k_j$ is the wavenumber corresponding to the layer domain $\Omega_j$. We note that the transmission operators given in equation~\eqref{eq:layered_pm} correspond to semi-infinite, and not bounded layers. As such, the width of the layers is not incorporated in the definition of the transmission operators defined in equation~\eqref{eq:layered_pm}.

It is also possible to employ the high-order shape deformation technique to construct transmission operators that are approximations of DtN operators corresponding to bounded periodic layers. Indeed, in the case of a bounded layer middle layer domain $\Omega_j$ we consider the boundary value problem
\begin{eqnarray}\label{eq:Hl}
  \Delta v_j+k_j^2v_j&=&0\quad{\rm in}\ \Omega_j^{per}\\
  v_j&=&g_{j,j-1}\quad{\rm on}\ \Sigma_{j,j-1}\nonumber\\
  v_j&=&g_{j,j+1}\quad{\rm on}\ \Sigma_{j,j+1}\nonumber
\end{eqnarray}
for which we define the DtN operator $\mathbf{Y}_j(k_j)\begin{bmatrix}g_{j,j-1}\\g_{j,j+1}\end{bmatrix}:=\begin{bmatrix}\partial_{n_j}v_j|_{\Sigma_{j,j-1}}\\ \partial_{n_j}v_j|_{\Sigma_{j,j+1}}\end{bmatrix}$. We mention that (a) the DtN operators $\mathbf{Y}_j(k_j)$ are $2\times 2$ matrix operators $\mathbf{Y}_j(k_j)=\begin{bmatrix}Y_{j-1,j-1}(k_j) & Y_{j-1,j+1}(k_j) \\ Y_{j+1,j-1}(k_j) & Y_{j+1,j+1}(k_j) \end{bmatrix}$, and (b) the same DtN operators are not properly defined for all wavenumbers $k_j$. Assuming that $F_{j-1}(x_1)=\varepsilon \widetilde{F_{j-1}}(x_1)$ and respectively $F_{j}(x_1)=\varepsilon \widetilde{F_{j}}(x_1)$ where $\widetilde{F_{j-1}}(x_1)$ and $\widetilde{F_{j}}(x_1)$ are analytic, the DtN operator $\mathbf{Y}_j(k_j)$ can be expressed in terms of the perturbation series
\begin{equation}\label{eq:pert_series_Yj}
\mathbf{Y}_j(k_j)=\sum_{\ell=0}^\infty \mathbf{Y}_{j,\ell}(k_j)\varepsilon^\ell.
\end{equation}
The operator terms in the perturbation series above are defined recursively in the following manner~\cite{nicholls2012three}
\begin{eqnarray}\label{eq:recursiveYj}
  \mathbf{Y}_{j,0}(k_j)&=&i\beta_D(k_j)\begin{bmatrix}\coth(ih_j\beta_D(k_j))&-\csch(ih_j \beta_D(k_j))\\-\csch(ih_j\beta_D(k_j)) & \coth(ih_j\beta_D(k_j))\end{bmatrix}\\
  \mathbf{Y}_{j,n}(k_j)&=&-(\mathbf{C}_n(\widetilde{F}_{j-1})+\mathbf{C}_n(\widetilde{F}_{j}))\frac{k_j^2}{i\beta_D(k_j)}-D(\mathbf{C}_n(\widetilde{F}_{j-1})+\mathbf{C}_n(\widetilde{F}_{j}))\frac{1}{i\beta_D(k_j)}D\nonumber\\
  &-&\sum_{m=0}^{n-1}\mathbf{Y}_{j,m}(k_j)\left[\mathbf{S}_{n-m}(\widetilde{F}_{j-1})+\mathbf{S}_{n-m}(\widetilde{F}_{j})\right]
  \end{eqnarray}
where $h_j=\overline{F}_{j-1}-\overline{F}_j$ and
\[
\mathbf{C}_n(\widetilde{F}_{j-1}):=\widetilde{F}_{j-1,n}\begin{bmatrix}\shch_{n+1}(ih_j\beta_D(k_j)) & (-1)^{n+1}\shch_{n+1}(0)\\ 0 & 0\end{bmatrix}\frac{(i\beta_D(k_j))^n}{\sinh(ih_j\beta_D(k_j))}
\]
\[
\mathbf{C}_n(\widetilde{F}_{j}):=\widetilde{F}_{j,n}\begin{bmatrix}0 & 0\\ -\shch_{n+1}(0) & (-1)^n\shch_{n+1}(ih_j\beta_D(k_j))\end{bmatrix}\frac{(i\beta_D(k_j))^n}{\sinh(ih_j\beta_D(k_j))}
\]
as well as
\[
\mathbf{S}_n(\widetilde{F}_{j-1}):=\widetilde{F}_{j-1,n}\begin{bmatrix}\shch_{n}(ih_j\beta_D(k_j)) & (-1)^{n}\shch_{n}(0)\\ 0 & 0\end{bmatrix}\frac{(i\beta_D(k_j))^n}{\sinh(ih_j\beta_D(k_j))}
\]
\[
\mathbf{S}_n(\widetilde{F}_{j}):=\widetilde{F}_{j,n}\begin{bmatrix}0 & 0\\ \shch_{n}(0) & (-1)^n\shch_{n}(ih_j\beta_D(k_j))\end{bmatrix}\frac{(i\beta_D(k_j))^n}{\sinh(ih_j\beta_D(k_j))}
\]
where $$\shch_n(z)=\frac{e^z-(-1)^ne^{-z}}{2}.$$
\begin{remark}\label{stab_rec}
  We note that recursions~\eqref{eq:recursiveYj} can be carried out in a straightforward manner in the Fourier space. However, unlike recursions~\eqref{eq:recursive_1}, the recursions~\eqref{eq:recursiveYj} do not avoid subtractive cancellations, and, as such, are prone to instabilities for rougher profiles $\widetilde{F}_{j-1}$ and $\widetilde{F}_j$. In order to bypass these instabilities, an alternative strategy based on changes of variables that straighten out the boundaries is proposed in~\cite{hong2017stable} for robust perturbative evaluations of layer DtN. Nevertheless, the latter strategy requires numerical solutions for the evaluation of the terms in the perturbation series of the DtN operators $\mathbf{Y}_j(k_j)$. As such, the evaluation of the DtN operators $\mathbf{Y}_j(k_j)$ via the straightening of boundaries strategy in~\cite{hong2017stable} becomes more involved than the straightforward one given by the recursions~\eqref{eq:recursiveYj}. It is our conviction that optimized transmission operators ought to be simple to implement in order to warrant their incorporation in DD algorithms. Consequently, we advocate for the use of the simple recursions~\eqref{eq:recursiveYj} to construct approximations of DtN operators, and we point out their limitations in the case of rough profiles. 
  \end{remark}
Again, the complexification of the wavenumber $\kappa_j=k_j+i\sigma_j,\sigma_j>0$ leads to corresponding Fourier multipliers $\mathbf{Y}_{j,0}(\kappa_j)$ (and thus $\mathbf{Y}_{j,n}(\kappa_j),1\leq n$) that are well defined for all values $h_j$. Therefore, we define the $2\times 2$ matrix operators
\[
\mathbf{Y}_j^L(\kappa_j)=\sum_{\ell=0}^L \mathbf{Y}_{j,\ell}(\kappa_j)\varepsilon^\ell=\begin{bmatrix}Y^L_{j-1,j-1}(\kappa_j) & Y^L_{j-1,j+1}(\kappa_j) \\ Y^L_{j+1,j-1}(\kappa_j) & Y^L_{j+1,j+1}(\kappa_j) \end{bmatrix},\ \kappa_j=k_j+i\sigma_j,\ \sigma_j>0
\]
and we select the transmission operators corresponding to the layer $\Omega_j,\ 1\leq j\leq N$ in the form
\begin{equation}\label{eq:trans_Zj}
  Z_{j,j-1}^L:=Y^L_{j-1,j-1}(\kappa_j),\ 1\leq j\leq N-1,\qquad Z^L_{j,j+1}:=Y^L_{j+1,j+1}(\kappa_j),\ 1\leq j\leq N
\end{equation}
as well as
\begin{equation}\label{eq:f_two_ops}
  Z_{0,1}^L:=Y^{+}_{L}(\kappa_0,\widetilde{F}_0)\qquad Z_{N+1,N}^L:=Y^{-}_{L}(\kappa_{N+1},\widetilde{F}_N).
  \end{equation}
Again, under the assumption that the shape perturbation parameter $\varepsilon$ is small enough (and in particular smaller than the radius of convergence of the perturbation series~\eqref{eq:pert_series_Yj}), the arguments in the proof of Lemma~\ref{coerc1} can be easily adapted to derive coercivity properties of the type~\eqref{eq:well-pos_t} for the transmission operators $Z_{j,j-1}^L$ and $Z^L_{j,j+1}$ defined in equation~\eqref{eq:trans_Zj}.

Finally, the transmission operators $Z_{j,j+1}^\flat$ corresponding to the DD with strip subdomains~\eqref{DDM_t_final} are simply selected to be  complexified versions of half-space DtN operators, that is
\begin{equation}\label{eq:Zflat_choice}
Z_{j,j+1}^\flat=Z_{j+1,j}^\flat:=-i\beta_D(\kappa_j),\quad 0\leq j\leq N.
\end{equation}
In what follows, we refer to the DD formulations~\eqref{DDM_t} and respectively~\eqref{DDM_t_final} corresponding to the choice of transmission operators presented in this section as quasi-optimal DD (QO DD). We refer in what follows to the operator QO DD matrix~\eqref{eq:mA} corresponding to the choice of transmission operators given on equation~\eqref{eq:layered_pm} by the acronym $\mathcal{A}^s$,  and to the one corresponding to transmission operators~\eqref{eq:trans_Zj} by $\mathcal{A}$. In the next section we derive explicit formulas for calculations of RtR operators associated with the DD formulations~\eqref{DDM_t} and respectively~\eqref{DDM_t_final} based on robust quasi-periodic boundary integral equations.

\section{Calculations of RtR operators in terms of boundary integral operators associated with quasi-periodic Green functions~\label{rtr}}
\subsection{Quasi-periodic Green functions, layer potentials and integral operators}

Clearly, at the heart of the implementation of DD are computations of their associated RtR maps. We present in this Section explicit representations of RtR maps in terms of boundary integral operators associated with quasi-periodic Green functions that will serve as the basis of the implementation of the DD formulations considered in this text. For a given wavenumber $k$, we define the $\alpha-$quasi-periodic Green function
\begin{equation}\label{eq:qper_G}
  G^q_k(x_1,x_2)=\sum_{n\in\mathbb{Z}} e^{-i\alpha nd}G_k(x_1+nd,x_2)
\end{equation}
where $G_k(x_1,x_2)=\frac{i}{4}H_0^{(1)}(k|\mathbf{x}|),\ \mathbf{x}=(x_1,x_2)$. We also define $\alpha_r:=\alpha+\frac{2\pi}{d}r$ and $\beta_{r}=(k^2-\alpha_r^2)^{1/2}$, where the branch of the square roots in the definition of $\beta_{r}$ is chosen in such a way that $\sqrt{1}=1$, and that the branch cut coincides with the negative imaginary axis. The series in the definition of the Green function $G^q_k$ in equation~\eqref{eq:qper_G} converges for wavenumbers $k$ for which none of the coefficients $\beta_r$ is equal to zero---that is wavenumber which are not Wood frequencies. In the case of wavenumber $k$ that is a Wood frequency, shifted quasi-periodic Green functions can be used instead~\cite{bruno2017three,perez2018domain}. 

We assume that the interface $\Gamma$ is defined as $\Gamma:=\{(x_1,F(x_1)):0\leq x_1\leq d\}$ where $F$ is a $C^2$ periodic function of principal period equal to $d$. Given a density $\varphi$ defined on $\Gamma$ (which can be extended by $\alpha$-quasiperiodicity to arguments $(x_1,F(x_1)), x_1\in\mathbb{R}$) we define the single and double layer potentials corresponding to a wavenumber $k$
\begin{equation}\label{eq:single_layer}
  [SL_k\varphi](\mathbf{x}):=\int_{\Gamma}G^q_k(\mathbf{x}-\mathbf{y})\varphi(\mathbf{y})ds(\mathbf{y})\quad [DL_k\varphi](\mathbf{x}):=\int_{\Gamma}\frac{\partial G^q_k(\mathbf{x}-\mathbf{y})}{\partial \mathbf{n}(\mathbf{y})}\varphi(\mathbf{y})ds(\mathbf{y})
\end{equation}
for $\mathbf{x}\notin\Gamma$ and $\mathbf{x}=(x_1,x_2)$ such that $0\leq x_1\leq d$. It is immediate to see that the quantities $[SL_k\varphi](\mathbf{x})$ and $[DL_k\varphi](\mathbf{x})$ are $\alpha$-quasiperiodic outgoing solutions of the Helmholtz equation corresponding to wavenumber $k$ in the domains $\{\mathbf{x}:x_2>F(x_1)\}$ and $\{\mathbf{x}:x_2<F(x_1)\}$ respectively. The Dirichlet and Neumann boundary values of the single and double layer potentials give rise to the four boundary integral operators associated with quasiperiodic Helmholtz problems. Denoting by $n(\mathbf{x})=(-F'(x_1),1),\ \mathbf{x}=(x_1,F(x_1)),\ 0\leq x_1\leq d$ the (non-unit) normal to $\Gamma$ pointing into the domain $\{\mathbf{x}:x_2>F(x_1)\}$ we define the single layer boundary integral operator
\begin{equation}\label{eq:SL}
 [S_k(\varphi)](\mathbf{x}):= \lim_{\varepsilon\to 0} [SL_k\varphi](\mathbf{x}\pm \varepsilon n(\mathbf{x}))=\int_{0}^dG_k^q(x-y,F(x)-F(y))\varphi((y,F(y)))\ |\mathbf{y}|'dy,
\end{equation}
with $\mathbf{x}=(x,F(x))$ and $\mathbf{y}=(y,F(y))$. Similarly, we also define the {\em weighted} single layer operator in the form
\begin{equation}\label{eq:SLw}
  [S_k^w(\varphi)](\mathbf{x}):=\int_{0}^dG_k^q(x-y,F(x)-F(y))\varphi((y,F(y)))dy,\quad \mathbf{x}=(x,F(x)).
\end{equation}
We further define
\begin{equation}\label{eq:DLT}
  \lim_{\varepsilon\to 0} \nabla[SL_k\varphi](\mathbf{x}\pm \varepsilon n(\mathbf{x}))\cdot n(\mathbf{x})=\mp \frac{1}{2}\varphi(\mathbf{x})\ |\mathbf{x}|'+[(K_k)^\top(\varphi)](\mathbf{x}),\quad \mathbf{x}=(x,F(x)).
\end{equation}
In equation~\eqref{eq:DLT}, the adjoint double layer operator can be expressed explicitly as
\begin{equation}\label{eq:DLT_explicit}
  [(K_k)^\top(\varphi)](\mathbf{x})=\int_{\Gamma}\frac{\partial G_k^q(\mathbf{x}-\mathbf{y})}{\partial n(\mathbf{x})}\varphi(\mathbf{y})ds(y),\quad \mathbf{x}\in\Gamma.
\end{equation}
We also define a weighted version of the adjoint double layer operators in the form
\begin{equation}\label{eq:DLT_explicit1}
  [(K_k^w)^\top(\varphi)](\mathbf{x})=\int_{0}^d\frac{\partial G_k^q(\mathbf{x}-\mathbf{y})}{\partial n(\mathbf{x})}\varphi((y,F(y)))dy,\quad \mathbf{x}\in\Gamma.
\end{equation}
In addition, applying the same machinery to the double layer potentials we can define the double layer operator
\begin{equation}\label{eq:DL}
  \lim_{\varepsilon\to 0} [DL_k\varphi](\mathbf{x}\pm \varepsilon n(\mathbf{x}))=\pm \frac{1}{2}\varphi(\mathbf{x})+[(K_k)(\varphi)](\mathbf{x}),\quad \mathbf{x}=(x,F(x))
\end{equation}
as well as the hypersingular operators
\begin{equation}\label{eq:N}
  \lim_{\varepsilon\to 0} \nabla[DL_k\varphi](\mathbf{x}\pm \varepsilon n(\mathbf{x}))\cdot\mathbf{n}(\mathbf{x})=[(N_k)(\varphi)](\mathbf{x}),\quad \mathbf{x}=(x,F(x)).
\end{equation}
Weighted versions of the double layer and hyper singular operators are defined accordingly~\cite{dominguez2016well}. In what follows we express RtR operators associated with quasiperiodic Helmholtz problems using the boundary integral operators introduced above.

\subsection{Boundary integral  representation of RtR maps and invertibility of the DDM formulation, one interface\label{I_RtR0}}

We start with the analysis of the case of one interface $\Gamma_0$ separating two semi-infinite domains under the assumption that $\Gamma_0$ is the graph of an analytic and periodic function. The motivation for this is that the particularly simple case of one interface already contains the main difficulties related to the analysis of the well-posedness of QO DD. Our analysis relies on the mapping properties of the RtR operators, which, in turn, are derived from boundary integral operator representations of the latter. Using the $\alpha$-quasiperiodic boundary integral operators above, we are now in the position to compute the RtR operators $\mathcal{S}^j,j=0,1$ corresponding to the semi-infinite domains $\Omega_j, j=0,1$. We note that in this case the operators $Z_{0,1}^{s,L}$ coincide with the operators $Z_{0,1}^L$, and the operators $Z_{1,0}^{s,L}$ coincide with the operators $Z_{1,0}^L$. We start with the calculation of RtR operator $\mathcal{S}^0$ corresponding to problem~\eqref{eq:H0} by seeking its solution $w_0$ in the form
\[
w_0(\mathbf{x}):=[SL_k\varphi_0](\mathbf{x}),\quad \mathbf{x}\notin\Gamma_0.
\]
We immediately obtain the following explicit formula for the RtR operator $\mathcal{S}^0$~\eqref{eq:S0}:
\begin{equation}\label{eq:calc_S0}
  \mathcal{S}^0=I-(Z_{0,1}^L+Z_{1,0}^L)S_{\Gamma_0,k_0}^w\left(\frac{1}{2}I+(K_{\Gamma_0,k_0}^w)^\top+Z_{1,0}^LS_{\Gamma_0,k_0}^w\right)^{-1},
\end{equation}
where the operators $(K_{\Gamma_0,k_0}^w)^\top$ are defined just as in equations~\eqref{eq:DLT_explicit1} but with normal $n_0$ pointing into $\Omega^{-}_0$ (exterior of $\Omega_0$). Here and in what follows we introduce an additional subscript to make explicit the curve that is the domain of integration of the boundary integral operators. Our next goal is to establish the robustness of the formulation~\eqref{eq:calc_S0}. We assume in what follows that the parameter $\varepsilon$ in the shape $\Gamma_0$ given by $F_0(x_1)=\varepsilon \widetilde{F}_0(x_1)$ is smaller than the minimum of the radii $\rho_j$ of convergence of the boundary perturbation expansion series of the DtN operators $Y^\pm(k_j,F_0)$
\[
Y^\pm(k_j,F_0)=\sum_{n=0}^\infty Y_{n}^\pm(k_j,\widetilde{F}_0)\varepsilon^n,\ j=0,1.
\]
We establish the following
\begin{theorem}\label{inv_A}
Assuming that the profile function $\widetilde{F_0}(x_1)$ is periodic and analytic, and the shape parameter $\varepsilon$ is small enough, the operator
\[
\mathcal{A}_{1,0}:=\frac{1}{2}I+(K_{\Gamma_0,k_0}^w)^\top+Z_{1,0}^LS_{\Gamma_0,k_0}^w:H^{-1/2}(\Gamma_0)\to H^{-1/2}(\Gamma_0)
\]
is invertible with continuous inverse.
\end{theorem}
\begin{proof}
Assuming that $\varepsilon<\rho_1$, where $\rho_1$ is the radius of convergence of the shape perturbation series of the DtN operator $Y^{-}(k_1,F_0)$, we have~\cite{nicholls2004shape}
\begin{equation}\label{eq:mapping_1}
\|Y^{-}(k_1,F_0)-\sum_{\ell=0}^L Y_{\ell}^{-}(k_1,\widetilde{F}_0)\varepsilon^\ell\|_{H^{1/2}(\Gamma_0)\to H^{-1/2}(\Gamma_0)}\lesssim \varepsilon^{L+1}.
\end{equation}
Given that
\[
|(k_1^2-\alpha_n^2)^{1/2}-((k_1+i\sigma_1)^2-\alpha_n^2)^{1/2}|=\mathcal{O}(n^{-1}),\ n\to \infty
\]
it follows that
\[
Y_0^{-}(k_1,\widetilde{F}_0)-Y_0^{-}(\kappa_1,\widetilde{F}_0)=-i(\beta_D(k_1)-\beta_D(\kappa_1)):H^{1/2}(\Gamma_0)\to H^{3/2}(\Gamma_0).
\]
Using the stable commutator representations~\eqref{eq:T1} and~\eqref{eq:T2}, together with the mapping properties of the commutators established in~\cite{nicholls2004shape}, we obtain
\[
Y_1^{-}(k_1,\widetilde{F}_0)-Y_1^{-}(\kappa_1,\widetilde{F}_0):H^{1/2}(\Gamma_0)\to H^{3/2}(\Gamma_0)
\]
and respectively
\[
Y_2^{-}(k_1,\widetilde{F}_0)-Y_2^{-}(\kappa_1,\widetilde{F}_0):H^{1/2}(\Gamma_0)\to H^{1/2}(\Gamma_0).
\]
In conclusion, we have 
\[
\sum_{\ell=0}^L Y_{\ell}^{-}(k_1,\widetilde{F}_0)\varepsilon^\ell-Z_{1,0}^L=Z_{1,0}^{L,0}+Z_{1,0}^{L,1}
\]
where $Z_{1,0}^{L,0}:H^{1/2}(\Gamma_0)\to H^{3/2}(\Gamma_0)$, and $Z_{1,0}^{L,1}=0$ for $L\leq 1$ and
$\|Z_{1,0}^{L,1}\|_{H^{1/2}(\Gamma_0)\to H^{1/2}(\Gamma_0)}\lesssim \varepsilon^{2}$ for $L=2$. Consequently, we can express the operator $\mathcal{A}_{1,0}$ in the form
\[
\mathcal{A}_{1,0}=I+\mathcal{A}_{1,0}^0+\mathcal{K}_{1,0}, L\leq 2
\]
where 
\[
\|\mathcal{A}_{1,0}^0\|_{H^{-1/2}(\Gamma_0)\to H^{-1/2}(\Gamma_0)}\lesssim \varepsilon^{2}
\]
and 
\[
\mathcal{K}_{1,0}:H^{-1/2}(\Gamma_0)\to H^{1/2}(\Gamma_0).
\]
In conclusion, the operator $\mathcal{A}_{1,0}:H^{-1/2}(\Gamma_0)\to H^{-1/2}(\Gamma_0)$ is a compact perturbation the operator $I+\mathcal{A}_{1,0}^0$, and the latter is invertible in the space $H^{-1/2}(\Gamma_0)$ provided that $\varepsilon$ is small enough. The invertibility of the operator $\mathcal{A}_{1,0}$ can be then established then via the Fredholm theory provided that the same operator is injective. The latter, in turn, follows from classical arguments and relies on the coercivity of the operator $Z_{1,0}^L$ established in Lemma~\ref{coerc1}.
\end{proof}

An immediate corollary of the results established in Theorem~\ref{inv_A} is a representation of the RtR operator $\mathcal{S}^0$ in the form
\begin{equation}\label{eq:repr_S0}
\mathcal{S}^0=\mathcal{S}^0_0+\mathcal{K}^0_0
\end{equation}
where $\mathcal{K}^0_0:H^{-1/2}(\Gamma_0)\to H^{1/2}(\Gamma_0)$ and thus $\mathcal{K}_0^0$ is a compact operator in $H^{-1/2}(\Gamma_0)$, and $\mathcal{S}^0_0:H^{-1/2}(\Gamma_0)\to H^{-1/2}(\Gamma_0)$ has a small norm
\begin{equation}\label{eq:smallness}
\|\mathcal{S}^0_0\|_{H^{-1/2}(\Gamma_0)\to H^{-1/2}(\Gamma_0)}\lesssim \varepsilon^{2}.
\end{equation}
A similar result can be established for the representation of the RtR operator $\mathcal{S}^1$ in the form $\mathcal{S}^1=\mathcal{S}^1_0+\mathcal{K}^1_0$, where the operators in the latter decomposition have the same mapping properties as those of the operators $\mathcal{S}^0_0$ and $\mathcal{K}^0_0$. We are now in the position to establish the well-posedness of the DD formulation in the case of one interface:
\begin{theorem}\label{eq:inv_DDM_1interface}
Assuming that he profile function $\widetilde{F_0}(x_1)$ is periodic and analytic, the QO DD operator matrix
\[
\mathcal{A}=\begin{bmatrix}I & \mathcal{S}^1\\ \mathcal{S}^0 & I\end{bmatrix}
\]
is invertible with continuous inverse in the space $H^{-1/2}(\Gamma_0)\times H^{-1/2}(\Gamma_0)$ provided that the shape parameter $\varepsilon$ is small enough.
\end{theorem}
\begin{proof}
  First, using the decompositions $\mathcal{S}^j=\mathcal{S}^j_0+\mathcal{K}^j_0,\ j=0,1$ where $\|\mathcal{S}^j_0\|_{H^{-1/2}(\Gamma_0)\to H^{-1/2}(\Gamma_0)}\lesssim \varepsilon^{2},\ j=0,1$ and $\mathcal{K}^j_0:H^{-1/2}(\Gamma_0)\to H^{-1/2}(\Gamma_0),\ j=0,1$ are compact, it follows that
  \[
  \mathcal{A}=\begin{bmatrix}I & \mathcal{S}^1_0\\ \mathcal{S}^0_0 & I\end{bmatrix}+\begin{bmatrix}0 & \mathcal{K}^0_0 \\ \mathcal{K}^1_0 & 0\end{bmatrix}.
  \]
  Neumann series arguments yield the fact that the matrix operator $\begin{bmatrix}I & \mathcal{S}^1_0\\ \mathcal{S}^0_0 & I\end{bmatrix}$ is invertible in the space $H^{-1/2}(\Gamma_0)\times H^{-1/2}(\Gamma_0)$, while the matrix operator $\begin{bmatrix}0 & \mathcal{K}^0_0 \\ \mathcal{K}^1_0 & 0\end{bmatrix}$ is compact in the same functional space $H^{-1/2}(\Gamma_0)\times H^{-1/2}(\Gamma_0)$. Consequently, the QO DD operator $\mathcal{A}$ is Fredholm of index zero in the space $H^{-1/2}(\Gamma_0)\times H^{-1/2}(\Gamma_0)$, and thus the result of the theorem is established once we prove the injectivity of the operator $\mathcal{A}$. Now let $(\varphi_0,\varphi_1)\in Ker(\mathcal{A})$ and define $u_0$ and $u_1$ be $\alpha$-quasiperiodic radiative solutions of the following Helmholtz boundary value problems
      \begin{eqnarray*}
        \Delta u_0+k_0^2u_0&=&0\quad{\rm in}\ \Omega_0\\
        \partial_{n_0}u_0+Z_{1,0}^Lu_0&=&\varphi_0\quad {\rm on}\ \Gamma_0
      \end{eqnarray*}
      and
 \begin{eqnarray*}
        \Delta u_1+k_1^2u_1&=&0\quad{\rm in}\ \Omega_1\\
        \partial_{n_1}u_1+Z_{0,1}^Lu_1&=&\varphi_1\quad {\rm on}\ \Gamma_0.
 \end{eqnarray*}
 The requirement $(\varphi_0,\varphi_1)\in Ker(\mathcal{A})$ translates into the following system of equations on $\Gamma_0$
 \begin{eqnarray*}
   \partial_{n_0}u_0+Z_{1,0}^Lu_0&=&-\partial_{n_1}u_1+Z_{1,0}^Lu_1\\
   \partial_{n_1}u_1+Z_{0,1}^Lu_1&=&-\partial_{n_0}u_0+Z_{0,1}^Lu_0.
 \end{eqnarray*}
 Using the injectivity of the operator $Z_{1,0}^L+Z_{0,1}^L$, we obtain immediately that $u_0=u_1$ and $\partial_{n_0}u_0=-\partial_{n_1}u_1$ on $\Gamma_0$. Hence, $u_j=0$ in $\Omega_j$ for $j=0,1$, and in conclusion $\varphi_j=0$ on $\Gamma_0$ for $j=0,1$. 
\end{proof}
\begin{remark}\label{flat_interface}
We note that in the case when $\Gamma_0$ is flat, the RtR operators $\mathcal{S}^j, j=0,1$ are actually compact in the space $H^{-1/2}(\Gamma_0)$.
\end{remark}

We turn our attention to the analysis of the QO DD~\eqref{DDM_t} in the case of multiple interfaces separating several layers.

\subsection{Boundary integral representation of RtR maps corresponding to the subdomains $\Omega^{per}_j$ and the invertibility of the QO DD formulation~\eqref{DDM_t}  in the case of several subdomains\label{I_RtR1}}

We begin by expressing the RtR operators $\mathcal{S}^j$ defined in equation~\eqref{RtRboxj_t} via boundary integral operators. We present our derivations in the case of transmission operators $Z^{s,L}_{j-1,j}$ and $Z^{s,L}_{j+1,j}$ defined in equation~\eqref{eq:layered_pm}; analogous results can be established in the case of transmission operators $Z_{j-1,j}^L$ and $Z_{j+1,j}^L$ defined in equation~\eqref{eq:trans_Zj}. We note that the Helmholtz problems~\eqref{eq:H} can be all expressed in the generic form
\begin{eqnarray}\label{eq:Hgen}
  \Delta w+k^2w&=&0\quad{\rm in}\ \Omega^{per}\\
  \partial_{n}w+Z_{j-1,j}^{s,L}\ w&=&g_{t}\quad{\rm on}\ \Gamma_{t}\nonumber\\
  \partial_{n}w+Z_{j+1,j}^{s,L}w&=&g_{b}\quad{\rm on}\ \Gamma_{b}\nonumber
  \end{eqnarray}
where $g_t,g_b$ are $\alpha$-quasiperiodic functions. Thus, the RtR operators $\mathcal{S}^j,1\leq j<N$~\eqref{RtRboxj_t} are all related to the following RtR operator associated with the Helmholtz problems~\eqref{eq:Hgen}:
\begin{equation}\label{RtRboxj_O}
   \mathcal{S}\begin{bmatrix}g_{t}\\g_{b}\end{bmatrix}=\begin{bmatrix}\mathcal{S}_{t,t} & \mathcal{S}_{t,b}\\ \mathcal{S}_{b,t}& \mathcal{S}_{b,b}\end{bmatrix}\begin{bmatrix}g_{t}\\g_{b}\end{bmatrix}:=\begin{bmatrix}(\partial_{n}w-Z_{j,j-1}^{s,L} w)|_{\Gamma_{t}}\\(\partial_{n}w-Z_{j,j+1}^{s,L} w)|_{\Gamma_{b}}\end{bmatrix}.
 \end{equation}
Seeking for the solution $w$ of equations~\eqref{eq:Hgen} in the form
\[
w = SL_{k,t}\varphi_t+SL_{k,b}\varphi_b
\]
where $SL_{k,t}^q$ ($SL_{k,b}^q$) denotes the quasiperiodic single layer potential whose domain of integration in $\Gamma_t$ ($\Gamma_b$), we arrive at the following expression for the RtR operator $\mathcal{S}$: 
\begin{eqnarray}\label{eq:S_layer_j}
  \mathcal{S}=\begin{bmatrix}I &0\\0& I\end{bmatrix}&-&\begin{bmatrix}Z_{j-1,j}^{s,L}+Z_{j,j-1}^{s,L}& 0\\0 & Z_{j+1,j}^{s,L}+Z_{j,j+1}^{s,L}\end{bmatrix}\begin{bmatrix}S^w_{k,t,t} & S^w_{k,b,t}\\ S^w_{k,t,b} & S^w_{k,b,b}\end{bmatrix}\nonumber\\
  &\times&\begin{bmatrix}1/2I+(K_{k,t,t}^w)^\top +Z_{j-1,j}^{s,L} S^w_{k,t,t} & (K^w_{k,b,t})^\top+Z_{j-1,j}^{s,L} S^w_{k,b,t}\\ (K_{k,t,b}^w)^\top+Z_{j+1,j}^{s,L} S^w_{k,t,b} & 1/2I+(K_{k,b,b}^w)^\top +Z_{j+1,j}^{s,L} S^w_{k,b,b}\end{bmatrix}^{-1}.
\end{eqnarray}
We note that in equation~\eqref{eq:S_layer_j}, the subscripts in the notation $S_{k,b,t}^w$ signify that in equation~\eqref{eq:SL} the target point $\mathbf{x}\in\Gamma_t$ and the integration point $\mathbf{y}\in\Gamma_b$, whereas in the notation $S_{k,b,t}^w$ signify that in equation~\eqref{eq:SL} the target point $\mathbf{x}\in\Gamma_t$ and the integration point $\mathbf{y}\in\Gamma_b$; all the other additional subscripts in equation~\eqref{eq:S_layer_j} have similar meanings related to the locations of target and integration points for single and adjoint double layer boundary integral operators. The invertibility of the operators featured in equation~\eqref{eq:S_layer_j} can be established using similar reasoning to that in the proof of Theorem~\ref{inv_A} under similar assumptions on the analyticity of the profiles $g_t$ and $g_b$, and the smallness of the shape perturbation parameter $\varepsilon$. Also, the operators in the block decomposition in equation~\eqref{RtRboxj_O} have the following mapping properties:
\[
\mathcal{S}_{t,b}:H^{-1/2}(\Gamma_b)\to H^{-1/2}(\Gamma_t),\quad \mathcal{S}_{b,t}:H^{-1/2}(\Gamma_t)\to H^{-1/2}(\Gamma_b)\quad\rm{are\ compact}
\]
and
\[
\mathcal{S}_{t,t}=\mathcal{S}_{t,t}^0+\mathcal{S}_{t,t}^1,\quad \|\mathcal{S}_{t,t}^0\|_{H^{-1/2}(\Gamma_t)\to H^{-1/2}(\Gamma_t)}\lesssim \varepsilon^{2},\quad \mathcal{S}_{t,t}^1:H^{-1/2}(\Gamma_t)\to H^{-1/2}(\Gamma_t)\quad\rm{is\ compact}
\]
as well as
\[
\mathcal{S}_{b,b}=\mathcal{S}_{b,b}^0+\mathcal{S}_{b,b}^1,\quad \|\mathcal{S}_{b,b}^0\|_{H^{-1/2}(\Gamma_b)\to H^{-1/2}(\Gamma_b)}\lesssim \varepsilon^{2},\quad \mathcal{S}_{b,b}^1:H^{-1/2}(\Gamma_b)\to H^{-1/2}(\Gamma_b)\quad\rm{is\ compact}.
\]
We are now in the position to prove the following
\begin{theorem}\label{thm_inv_DD}
Assuming that the transmission problem~\eqref{system_t} is well-posed, the profiles  $\widetilde{F}_j(x)$ are all analytic for $0\leq j\leq N$, and that the shape parameter $\varepsilon$ corresponding to the grating profiles $F_j(x)=\varepsilon \widetilde{F}_j(x),\ 0\leq j\leq N$ is small enough. The QO DD operator matrix $\mathcal{A}$ defined in equation~\eqref{eq:mA} is invertible in the space $H^{-1/2}(\Gamma_0)\times H^{-1/2}(\Gamma_0)\times\ldots\times H^{-1/2}(\Gamma_N)\times H^{-1/2}(\Gamma_N)$.
\end{theorem}
\begin{proof} Assuming that $\varepsilon$ is smaller than all the radii of convergence of the shape perturbation series of the DtN operators $Y_0$, $Y_{N}$, and $\mathbf{Y}_j,\ 1\leq j\leq N$, we use the results established above to 
\[
\mathcal{S}^j_{\ell,\ell}=\mathcal{S}^{j,0}_{\ell,\ell}+\mathcal{S}^{j,1}_{\ell,\ell},\quad \ell=j-1,j+1
\]
where
\[
 \|\mathcal{S}_{\ell,\ell}^{j,0}\|_{H^{-1/2}(\Gamma_\ell)\to H^{-1/2}(\Gamma_\ell)}\lesssim \varepsilon^{2},\quad \mathcal{S}_{\ell,\ell}^{j,1}:H^{-1/2}(\Gamma_\ell)\to H^{-1/2}(\Gamma_\ell)\quad\rm{is\ compact}.
\]
Similar decomposition can be effected on the RtR operators $\mathcal{S}^0$ and respectively $\mathcal{S}^{N+1}$. Then, we can express the QO DD operator $\mathcal{A}$ in the form
\[
\mathcal{A}=\mathcal{A}^0+\mathcal{A}^1
\]
where
\begin{equation}\label{eq:mAC}
  \mathcal{A}^0:=\begin{bmatrix}
  I & \mathcal{S}^{1,0}_{0,0}& 0 & 0 & \ldots & 0 & 0 & 0 & 0 & \ldots & 0 & 0 \\
  \mathcal{S}^{0,0}_{1,1} & I & 0 & 0 & \ldots & 0 & 0 & 0 & 0 &  \ldots & 0 & 0\\
  0 & 0 & I & \mathcal{S}^{2,0}_{1,1} & \ldots & 0 & 0 & 0 & 0 &  \ldots & 0 & 0\\
  0 & 0 & \mathcal{S}^{1,0}_{2,2} & I & \ldots & 0 & 0 & 0 & 0 &  \ldots & 0 & 0\\
  \ldots & \ldots & \ldots & \ldots & \ldots & \ldots & \ldots & \ldots & \ldots & \ldots & \ldots\\
  0 & 0 & 0 & 0 & \ldots & 0 & I & \mathcal{S}^{j+1,0}_{j,j} & 0 & \ldots  & 0 & 0\\
  0 & 0 & 0 & 0 & \ldots & 0 & \mathcal{S}^{j,0}_{j+1,j+1} & I & 0 &  \ldots & 0 & 0\\
  \ldots & \ldots & \ldots & \ldots & \ldots & \ldots & \ldots & \ldots & \ldots & \ldots & \ldots & \ldots\\
  \ldots & \ldots & \ldots & \ldots & \ldots & \ldots & \ldots & \ldots & \ldots & \ldots & I & \mathcal{S}^{N+1,0}_{N,N}\\
 \ldots & \ldots & \ldots & \ldots & \ldots & \ldots & \ldots & \ldots & \ldots & \ldots & \mathcal{S}^{N,0}_{N+1,N+1} & I\\
  \end{bmatrix}.\nonumber
\end{equation}
and $\mathcal{A}^1$ is compact in the space $H^{-1/2}(\Gamma_0)\times H^{-1/2}(\Gamma_0)\times\ldots\times H^{-1/2}(\Gamma_N)\times H^{-1/2}(\Gamma_N)$. Given the bounds established above on the operators that are non-diagonal entries in the matrix operator $\mathcal{A}^0$, we conclude that the operator $\mathcal{A}^0$ is invertible in the space $H^{-1/2}(\Gamma_0)\times H^{-1/2}(\Gamma_0)\times\ldots\times H^{-1/2}(\Gamma_N)\times H^{-1/2}(\Gamma_N)$. Thus, the invertibility of the operator $\mathcal{A}$ is equivalent to its injectivity. The latter, in turn, follows from the well-posedness of the transmission problem~\eqref{system_t} just as in the proof of Theorem~\ref{inv_A}.
\end{proof}

\subsection{Boundary integral representation of RtR maps corresponding to strip subdomains and the invertibility of the QO DD formulation~\eqref{DDM_t_final}\label{I_RtR}}

We present a representation of the RtR operators associated with the Helmholtz transmission boundary value problem~\eqref{DDM_t_final_well_p}. We assume for simplicity that none of the wavenumbers $k_{j-1}$ and $k_j$ are Wood frequencies. Then, we look for the solution $v_j$ of the boundary value problem 
\begin{eqnarray*}
  \Delta v_j +k_j(x)^2 v_j &=&0\qquad {\rm in}\quad \Omega_j^{\flat,per},\\
  k_0(x):=k_0, k_j(x)&:=&\begin{cases}k_{j-1},\ x_2>\overline{F}_{j-1}+F_{j-1}(x_1)\\ k_j,\ x_2<\overline{F}_{j-1}+F_{j-1}(x_1)\end{cases},\nonumber\\
  \left[v_j\right]=0,\quad \left[\partial_{n_j}v_j\right]&=&0\quad{\rm on}\ \Gamma_{j-1}\nonumber\\
  \partial_{x_2}v_{j}+Z_{j-1,j}^\flat v_{j}&=&g^\flat_{j,j-1}\ {\rm on}\ \Sigma^\flat_{j,j-1},\nonumber\\
  -\partial_{x_2}v_{j}+Z_{j+1,j}^\flat v_{j}&=&g^\flat_{j,j+1}\ {\rm on}\ \Sigma^\flat_{j+1,j},\nonumber
\end{eqnarray*}
in the form
\begin{equation}\label{eq:repr_vj}
  v_j(\mathbf{x})=\begin{cases}[SL_{k_{j-1},\Sigma^\flat_{j,j-1}}\varphi_{j,j-1}](\mathbf{x})+[SL_{k_{j-1},\Gamma_{j-1}}\varphi](\mathbf{x})+[DL_{k_{j-1},\Gamma_{j-1}}\psi](\mathbf{x}),& x_2>\overline{F}_{j-1}+F_{j-1}(x_1)\\
  [SL_{k_{j},\Sigma^\flat_{j+1,j}}\varphi_{j+1,j}](\mathbf{x})+[SL_{k_{j},\Gamma_{j-1}}\varphi](\mathbf{x})+[DL_{k_{j},\Gamma_{j-1}}\psi](\mathbf{x}),& x_2<\overline{F}_{j-1}+F_{j-1}(x_1)
  \end{cases}
  \end{equation}
where the double layer potentials on the interface $\Gamma_{j-1}$ are defined with respect to the unit normal $n_{j-1}=\frac{(F'_{j-1}(x_1),-1)}{(1+(F'_{j-1}(x_1))^2)}$ pointing towards the domain $\Omega_j$. The enforcement of boundary conditions leads to the following system of BIEs

\begin{eqnarray}\label{eq:BIE_system_flat}
  \left(\frac{1}{2}I+K^\top_{k_{j-1},\Sigma^\flat_{j,j-1}}+Z^\flat_{j-1,j}S_{k_{j-1},\Sigma^\flat_{j,j-1}}\right)\varphi_{j,j-1}+
  \left(\partial_{x_2} SL_{k_{j-1},\Gamma_{j-1},\Sigma^\flat_{j-1,j}}+Z^\flat_{j-1,j}SL_{k_{j-1},\Gamma_{j-1},\Sigma^\flat_{j-1,j}}\right)\varphi\nonumber\\
    +\left(\partial_{x_2} DL_{k_{j-1},\Gamma_{j-1},\Sigma^\flat_{j-1,j}}+Z^\flat_{j-1,j}DL_{k_{j-1},\Gamma_{j-1},\Sigma^\flat_{j-1,j}}\right)\psi=g_{j,j-1}^\flat\nonumber\\
  \left(\frac{1}{2}I+K^\top_{k_{j},\Sigma^\flat_{j+1,j}}+Z^\flat_{j+1,j}S_{k_{j},\Sigma^\flat_{j+1,j}}\right)\varphi_{j+1,j}+\left( -\partial_{x_2} SL_{k_{j},\Gamma_{j-1},\Sigma^\flat_{j+1,j}}+Z^\flat_{j+1,j}SL_{k_{j},\Gamma_{j-1},\Sigma^\flat_{j+1,j}}\right)\varphi\nonumber\\
  +\left( -\partial_{x_2} DL_{k_{j},\Gamma_{j-1},\Sigma^\flat_{j+1,j}}+Z^\flat_{j+1,j}DL_{k_{j},\Gamma_{j-1},\Sigma^\flat_{j+1,j}}\right)\psi=g_{j+1,j}^\flat\nonumber\\
  \left(\partial_{n_j}SL_{k_{j-1},\Sigma^\flat_{j,j-1},\Gamma_{j-1}}\right)\varphi_{j,j-1} -\left(\partial_{n_j}SL_{k_{j},\Sigma^\flat_{j+1,j},\Gamma_{j-1}}\right)\varphi_{j+1,j}
  +\left( I + K^\top_{k_{j-1},\Gamma_{j-1}}-K^\top_{k_{j},\Gamma_{j-1}}\right)\varphi\nonumber\\
  +\left( N_{k_{j-1},\Gamma_{j-1}}-N_{k_{j},\Gamma_{j-1}}\right)\psi=0\nonumber\\
  -\left(SL_{k_{j-1},\Sigma^\flat_{j,j-1},\Gamma_{j-1}}\right)\varphi_{j,j-1} +\left(SL_{k_{j},\Sigma^\flat_{j+1,j},\Gamma_{j-1}}\right)\varphi_{j+1,j}
  +\left( S_{k_j,\Gamma_{j-1}}-S_{k_{j-1},\Gamma_{j-1}}\right)\varphi\nonumber\\
  +\left( I + K_{k_j,\Gamma_{j-1}}-K_{k_{j-1},\Gamma_{j-1}}\right)\psi=0\nonumber\\
\end{eqnarray}
which can be shown easily to be equivalent to the Helmholtz transmission problem~\eqref{DDM_t_final_well_p}.
In addition, it is relatively straightforward to show using the BIE~\eqref{eq:BIE_system_flat} that RtR operator $S^{\flat,j}$ associated with the Helmholtz boundary value~\eqref{DDM_t_final_well_p} and explicitly defined in equation~\eqref{RtRboxj_t_flat} are compact operators in the space $H^{-1/2}(\Sigma^\flat_{j,j-1})\times H^{-1/2}(\Sigma_{j+1,j}^\flat)$ under the assumption that the periodic function $F_{j-1}$ is $C^2$ or better. Thus, the block operators in the representation $S^{\flat,j}=\begin{bmatrix}\mathcal{S}^{\flat,j}_{j-1,j-1} & \mathcal{S}^{\flat,j}_{j-1,j+1}\\ \mathcal{S}^{\flat,j}_{j+1,j-1} & \mathcal{S}^{\flat,j}_{j+1,j+1}\end{bmatrix}$ are themselves compact operators in appropriate functional spaces. In conclusion, the DD operator $\mathcal{A}^\flat$ corresponding to the QO DD formulation~\eqref{DDM_t_final} is a compact perturbation of the identity. Thus, its invertibility can be established analogously to that of the DD operator in Theorem~\ref{thm_inv_DD} under the assumption that the original Helmholtz transmission problem~\eqref{system_t} is well-posed.  We note that the well-posedness of the QO DD formulation~\eqref{DDM_t_final} holds regardless of the roughness of the profiles $\Gamma_j$, as long as the flat interfaces do not intersect the interfaces of material discontinuity.

We present in the next section numerical results that showcase the iterative behavior of solvers based on the QO DD formulations considered in this paper. 

\section{Numerical results\label{num}}

\subsection{Nystr\"om discretization}

Our numerical solution of equations~\eqref{ddm_t_exp} and~\eqref{ddm_t_exp_final} relies on Nystr\"om discretizations of the boundary integral operators that feature in the computation of the RtR operators given in Section~\ref{rtr}. A key ingredient in the evaluation of quasi-periodic boundary integral operators is the efficient evaluation of quasi-periodic Green function $G^q_k$ defined in equation~\eqref{eq:qper_G}. For frequencies that are away from Wood frequencies, we employ the recently introduced Windowed Green Function Method~\cite{bruno2016superalgebraically,bruno2017three,Delourme}. Specifically, let $\chi(r)$ be a smooth cutoff function equal to $1$ for $r<1/2$ and equal to $0$ for $r>1$ and define the windowed Green functions
\begin{equation}\label{eq:qper_GA}
  G^{q,A}_k(x_,x_2)=\sum_{m\in\mathbb{Z}} e^{-i\alpha md}G_k(x_1+md,x_2)\chi(r_m/A),\quad r_m=((x_1+md)^2+x_2^2)^{1/2}.
\end{equation}
The functions $G^{q,A}_k$ converge superalgebraically fast to $G^q_k$ as $A\to\infty$ when $k$ is not a Wood frequency~\cite{bruno2016superalgebraically,bruno2017three,Delourme}. Consequently, we make use of the functions $G^{q,A}_k$ for large $A$ in the definition of the quasiperiodic boundary integral operators. In the case of wavenumber $k$ which are Wood frequencies, we use shifted Green functions and their associated boundary integral operators~\cite{Delourme}. Given that the functions $G^{q,A}_k$  exhibit the same singularities as the free-space Green's functions $G_k$, the four quasiperiodic boundary integral operators are discretized using  trigonometric collocation and the singular quadratures of Martensen-Kussmaul (MK) that rely on logarithmic splitting of the kernels~\cite{kusmaul,martensen}. The full description of these discretizations is provided in~\cite{Delourme}. Since the transmission operators considered in this paper are Fourier multipliers, their discretization is straightforward in the context of the trigonometric interpolation framework that is the basis of the Nystr\"om methods described above. 

In conclusion, using Nystr\"om discretizations of the boundary integral operators based on trigonometric interpolation with $n$ equi-spaced points, we produce $\mathbb{C}^{n\times n}$ Nystr\"om discretization matrices of the four quasi-periodic boundary integral operators. Using Nystr\"om discretization matrices of quasi-periodic boundary integral operators within the integral representations of the RtR operators presented in Sections~\ref{I_RtR0} and~\ref{I_RtR1} (cf. formulas~\eqref{eq:calc_S0} and~\eqref{eq:S_layer_j}), we obtain Nystr\"om discretization matrices $\mathcal{S}^{j,n}$ of the corresponding RtR operators $\mathcal{S}^j$ for various choices of transmission operators. For instance, in the case when the transmission operators $Z_{j,j-1}^L$ and $Z_{j,j+1}^L$ are defined as in equations~\eqref{eq:trans_Zj}, the RtR Nystr\"om discretization matrices $\mathcal{S}^{j,n}$ of the continuous RtR operators $\mathcal{S}^j$ are expressed in $\mathbb{C}^{n\times n}$ block form 
\begin{equation*}
   \mathcal{S}^{j,n}=\begin{bmatrix}\mathcal{S}^{j,n}_{j-1,j-1} & \mathcal{S}^{j,n}_{j-1,j}\\ \mathcal{S}^{j,n}_{j,j-1} & \mathcal{S}^{j,n}_{j,j}\end{bmatrix}.
\end{equation*}
We note that RtR representation formulas~\eqref{eq:calc_S0} and~\eqref{eq:S_layer_j} require inverting boundary integral operators. Inverting their Nystr\"om discretization matrices can be performed in practice via direct solvers (when warranted by the size of the problem) or more generally by iterative solvers such as GMRES. This procedure leads to the construction of a Nystr\"om discretization matrix of the continuous DD matrix $\mathcal{A}$ defined in equation~\eqref{eq:mA} which is expressed in the block form 
\begin{equation}\label{eq:mAn}
\mathcal{A}_n=\begin{bmatrix}
D_{0}^n & U_{0}^n & 0_n &\ldots & 0_n\\
L_{0}^n & D_{1}^n & U_{1}^n & \ldots & 0_n\\
\ldots & \ldots & \ldots & \ldots & \ldots\\
\ldots &L_{j-1}^n & D_{j}^n & U_{j}^n & \ldots\\
\ldots & \ldots & \ldots & \ldots & \ldots\\
\ldots & \ldots & L_{N-2}^n & D_{N-1}^n & U_{N-1}^n\\
\ldots & \ldots & \ldots & L_{N-1}^n & D_{N}^n\\
\end{bmatrix}
\end{equation}
where
\begin{equation}\label{eq:components_n}
D_{j}^n:=\begin{bmatrix} I_n & \mathcal{S}^{j+1,n}_{j,j} \\ \mathcal{S}^{j,n}_{j+1,j+1}& I_n \end{bmatrix}\quad U^n_j=\begin{bmatrix} \mathcal{S}^{j+1,n}_{j,j+2} & 0_n \\ 0_n & 0_n\end{bmatrix}\quad L_j^n=\begin{bmatrix}0_n & 0_n \\ 0_n & \mathcal{S}^{j+1,n}_{j+2,j} \end{bmatrix}.
\end{equation}
Similarly, the Nystr\"om discretization matrices of the DD operators~\eqref{eq:mA} corresponding to transmission operators $Z_{j,j-1}^{s,L}$ and $Z_{j,j+1}^{s,L}$ defined in equations~\eqref{eq:layered_pm} are denoted by $\mathcal{A}_n^s$; also, the Nystr\"om discretization matrix corresponding to the stripe subdomain DD formulation~\eqref{DDM_t_final} is denoted by $\mathcal{A}^\flat_n$. Neither of the matrices $\mathcal{A}_n$, $\mathcal{A}_n^s$, nor $\mathcal{A}^\flat_n$ are stored in practice. Instead, the solution of the discrete DD systems feturing these matrices is performed via Krylov subspace iterative solvers such as GMRES. The main scope of the numerical results presented in this paper is to present the performance of GMRES solvers involving the DD discretization matrices $\mathcal{A}_n$,  $\mathcal{A}_n^s$, and $\mathcal{A}^\flat_n$.

As it is well documented, the choice of the transmission operators in DD formulations of Helmholtz transmission problems is motivated by optimizing the exchange of information between adjacent layers/subdomains. However, for high-frequency/high-contrast periodic layered media, there is significant global exchange of information amongst all layers, which cannot be captured by local transmission operators alone, regardless of how optimized these transmission operators are. One widely used remedy to deal with the global inter-layer communication is based on sweeping preconditioners. Sweeping preconditioners achieve an approximate block LU factorization of the DD matrix $\mathcal{A}_n$ (or $\mathcal{A}_n^s$ and $\mathcal{A}^\flat_n$ for that matter). In the case of DD for layered media, the sweeping preconditioners can be easily constructed on the basis of a very elegant matrix interpretation~\cite{vion2014double} which we describe briefly next. The \emph{exact} LU factorization of the block tridiagonal matrix $\mathcal{A}_n$ takes on the form
\[
\mathcal{A}_n=\begin{bmatrix}T_0 & \ldots & \ldots & \ldots\\ L_0^n & T_1 &\ldots &\ldots\\ \ldots & \ldots & \ldots & \ldots\\ \ldots & L_{N-2}^n & T_{N-1} &0_n  \\ \ldots & \ldots & L_{N-1}^n & T_N\end{bmatrix} \ \begin{bmatrix}I_n & T_0^{-1}U_0^n & \ldots & \ldots \\ 0_n & I_n & T_1^{-1}U_1^n &\ldots \\ \ldots & \ldots & \ldots &\ldots \\ \ldots & \ldots & I_n & T_{N-1}^{-1}U_{N-1}^n\\ \ldots & \ldots & \ldots & I_n\end{bmatrix}
\]
where
\begin{eqnarray*}
T_0&=&D_0^n\\
T_j&=&D_j^n-L_{j-1}^n T_{j-1}^{-1}U_{j-1}^n,\ 1\leq j.
\end{eqnarray*}
An approximate LU factorization of the matrix $\mathcal{A}_n$ can be derived on the premise that optimal transmission operators ought to act like perfectly transparent boundary conditions. This would entail that the block operators $\mathcal{S}^{j,n}_{j+1,j+1}$ and $\mathcal{S}^{j,n}_{j-1,j-1}$ be identically zero, which means that all the diagonal blocks $D_j$ are approximated by the identity matrix $I_n$. Given that $L_{j-1}^n U_{j-1}^n=0_n$,  a very simple approximate LU factorization of the matrix $\mathcal{A}_n$ is provided by
\begin{equation}\label{eq:mBn}
\mathcal{A}_n\approx\mathcal{B}_n:=\begin{bmatrix}I_n & \ldots & \ldots & \ldots\\ L_0^n & I_n &\ldots &\ldots\\ \ldots & \ldots & \ldots & \ldots\\ \ldots & L_{N-2}^n & I_n &0_n  \\ \ldots & \ldots & L_{N-1}^n & I_n\end{bmatrix} \ \begin{bmatrix}I_n & U_0^n & \ldots & \ldots \\ 0_n & I_n & U_1^n &\ldots \\ \ldots & \ldots & \ldots &\ldots \\ \ldots & \ldots & I_n & U_{N-1}^n\\ \ldots & \ldots & \ldots & I_n\end{bmatrix}.
\end{equation}
 Clearly, computing the inverse $\mathcal{B}_n^{-1}$ (or actually solving $\mathcal{B}_nf_n=g_n$) is straightforward as it does not involve any inversions of (smaller) block matrices. Given that the computation of $\mathcal{B}_n^{-1}$ involves a downward sweep followed by an upward sweep, the use of $\mathcal{B}_n^{-1}$ as a preconditioner to $\mathcal{A}_n$ is referred to as double sweep preconditioning. 
 \begin{remark}
 The exact LU factorization can also be employed for the solution of QO DD linear systems involving the matrix $\mathcal{A}_n$ provided that the RtR discretization matrices $\mathcal{S}^{j,n},\ 0\leq j\leq N$ are constructed---see~\cite{perez2018domain} for details of such a direct DD approach. The iterative approach proposed in this paper is more flexible, as the RtR matrices  $\mathcal{S}^{j,n},\ 0\leq j\leq N$ need not be assembled. Furthermore, once we compute the action of the matrix $\mathcal{A}_n$ on a given  vector, matrix inversions are no longer needed. 
 \end{remark}
 We present in the next section of variety of numerical results that highlight the iterative behavior of the QO DD formulations of quasi-periodic transmission problems in layered media, as well as the effectiveness of the sweeping preconditioners in the case of large numbers of layers. 

\subsection{QO DD solvers and sweeping preconditioners}
  
  We present in this section various numerical examples that illustrate the iterative behavior of the QO DD solvers using sweeping preconditioners. We consider both smooth and Lipschitz grating profiles that exhibit various degrees of roughness---as measured by the ratio between the height and the period, as well as by the oscillatory nature of the profile. Specifically, we consider a simple sinusoidal grating profile $\Gamma_0$ described by $x_2=\varepsilon \widetilde{F}_0(x_1),\ \widetilde{F}_0(x_1):=2.5\cos{x_1}$ for various values of $\varepsilon$, and the subsequent profiles $\Gamma_\ell$ being simple down shifted versions of the first profile, that is the grating $\Gamma_\ell$ is given by $x_2=F_\ell(x_1), F_\ell(x_1):=-\ell H + 2.5\varepsilon\cos{x_1}, 0\leq \ell\leq N,\ 0<H$. We also consider rougher profiles described in the form $x_2=\varepsilon \widetilde{F}_0(x_1),\ \widetilde{F}_0(x_1)=\pi(0.4\cos(x_1)-0.2\cos(2x_1)+0.4\cos(3x_1))$, and subsequent gratings profiles $F_\ell(x_1)$ that are obtained by vertically shifting the profile $F_0(x_1)$. Clearly, the parameter $\varepsilon$ constitutes a measure of the interface roughness. In addition, we consider Lipschitz grating profiles depicted in Figure~\ref{fig:Lipschitz} of period $2\pi$ and height $\varepsilon$ (note that the second profile in Figure~\ref{fig:Lipschitz} is not the graph of a $2\pi$ periodic function). 

  In most of the numerical results in this section we report the numbers of iterations required by the QO DD solvers to reach relative GMRES residuals of $10^{-4}$. Specifically, we used GMRES to solve the linear systems corresponding to discretization matrices $\mathcal{A}_n$ and $\mathcal{A}_n^s$ corresponding to the QO DD formulation~\eqref{eq:mA} with transmission operators defined in equations~\eqref{eq:trans_Zj} and~\eqref{eq:layered_pm} respectively, and $\mathcal{A}_n^\flat$ corresponding to the QO DD formulation with stripe subdomains~\eqref{DDM_t_final}. We emphasize one more time that none of the aforementioned QO DD discretization matrices is stored in practice. We also specify in the table headers the various approximation orders $L$ in the definition of the transmission operators ~\eqref{eq:trans_Zj} and~\eqref{eq:layered_pm} that enter the QO DD formulations~\eqref{eq:mA}; for instance, we specify $\mathcal{A}_n/\mathcal{A}_n^s,\ L=0$ and respectively $\mathcal{A}_n/\mathcal{A}_n^s,\ L=2$. As previously discussed, higher values of the approximation parameter $L$ leads not only to more cumbersome calculations of transmission operators, but more importantly to significant subtractive cancellations. We also investigate the effectiveness of the sweeping preconditioner applied to the QO DD discretization matrices; in this case we present the numbers of GMRES iterations needed after application of the sweeping preconditioner under the table header $\mathcal{B}_n^{-1}\mathcal{A}_n$ (and its analogues $(\mathcal{B}_n^s)^{-1}\mathcal{A}_n^s$ and $(\mathcal{B}_n^\flat)^{-1}\mathcal{A}_n^\flat$). We emphasize that the structure of the matrices $\mathcal{B}_n$ cf. equation~\eqref{eq:mBn} (and also $\mathcal{B}_n^s$ and $\mathcal{B}_n^\flat$) allows for calculations of their inverses $\mathcal{B}_n^{-1}$ (as well as $(\mathcal{B}_n^s)^{-1}$ and $(\mathcal{B}_n^\flat)^{-1}$) that do not require inverting block submatrices.

In the numerical experiments presented here we chose discretization sizes $n$ and windowing parameters $A$ in the definition of the quasi-periodic Green function $G_k^{q,A}$ in equation~\eqref{eq:qper_GA} so that the solutions produced by DD based on Nystr\"om discretizations of the RtR operators exhibit accuracies of the order $10^{-4}$ (or better) as measured by conservation of energy metrics. Specifically, we selected the windowing parameter $A=120$ and the discretization size $n=256$ in all the results presented in this section, with the exception of the experiments involving perfectly conducting inclusions, where we chose a larger windowing parameter $A=300$. The Nystr\"om discretization matrices of the RtR maps were produced following the calculations presented in Section~\ref{rtr} using direct linear algebra solvers to invert the discretization matrices corresponding to boundary integral operators. In all the numerical results presented in this section we considered normal incidence, that is the quasi-periodic parameter $\alpha=0$. Qualitatively similar results are obtained for other values of $\alpha$. Finally, unless specified, the wavenumbers considered in the numerical experiments are not Wood frequencies.  
  
  \begin{figure}
\centering
\includegraphics[scale=0.8]{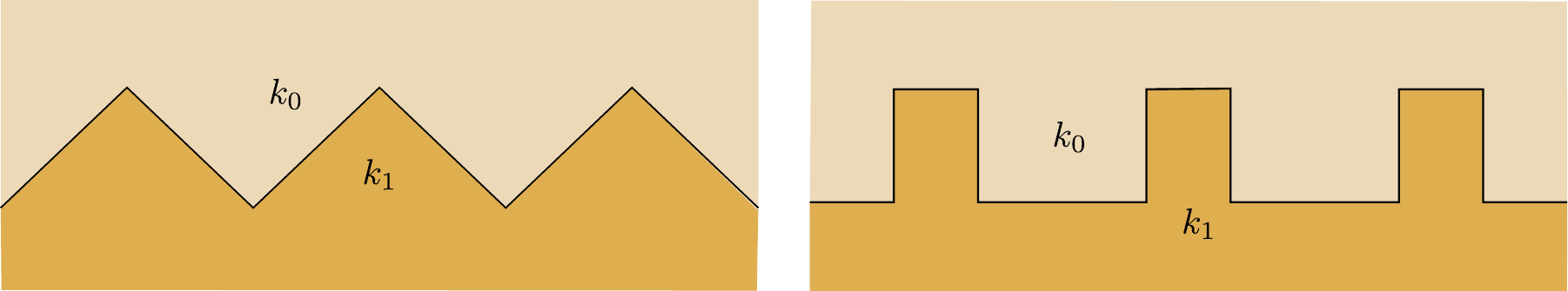}
\caption{Lipschitz grating profiles of period $2\pi$ and height $\varepsilon$.}
\label{fig:Lipschitz}
  \end{figure}

  \emph{Two layers} We start the presentation of our numerical results with the case of two semi-infinite layers separated by a periodic interface. We present in Table~\ref{comp1a} numbers of GMRES iterations required by the QO DD Nystr\"om discretization matrices $\mathcal{A}_n$ to reach GMRES relative residuals of $10^{-6} $ in the case of a deep, smooth grating interfaces separating two high-contrast media. Commensurate energy errors were produced by the Nystr\"om discretizations of the DD linear system. The wavenumbers considered in these results correspond to periodic transmission problems of periods that consist of $5,10,20$ and respectively $80$ wavelengths. We remark that using transmission operators $Z^2_{0,1}$ and $Z^2_{1,0}$ in the DD algorithm gives rise  to numbers of GMRES iterations that scale very mildly with respect to the increasing frequencies. We continue in Table~\ref{comp1aa} with numerical examples concerning a deep Lipschitz grating separating two high-contrast media. In the case of Lipschitz interfaces, we used transmission operators $Z^0_{0,1}$ and $Z^0_{1,0}$ respectively; we observed that the use of higher-order transmission operators $Z^L_{0,1}$ and $Z^L_{1,0}$ with $1\leq L\leq 2$ does not lead to improved iterative convergence of the QO DD solvers.
  \begin{table}
   \begin{center}
     \resizebox{!}{1.1cm}
{   
\begin{tabular}{|c|c|c|c|c|c|c|c|}
  \hline
  \multicolumn{2}{|c|} {$k_0=1.3, k_1=4.3$}& \multicolumn{2}{c|} {$k_0=2.3, k_1=8.3$}& \multicolumn{2}{c|} {$k_0=4.3, k_1=16.3$}& \multicolumn{2}{c|} {$k_0=16.3, k_1=64.3$}\\
\cline{1-8}
$\mathcal{A}_n,L=0$ & $\mathcal{A}_n,L=2$ & $\mathcal{A}_n, L=0$ & $\mathcal{A}_n,L=2$ & $\mathcal{A}_n, L=0$ & $\mathcal{A}_n,L=2$& $\mathcal{A}_n, L=0$ & $\mathcal{A}_n,L=2$\\
\hline
14 & 12 & 16 & 12 & 19 & 14 & 21 & 14\\
\hline
\hline
19 & 17 & 22 & 15 & 24 & 17 & 28 & 18\\
\hline
\end{tabular}
}
\caption{Numbers of GMRES iterations required by QO DD formulations to reach relative residuals of $10^{-6}$ for configurations consisting of $2$ layers, where the interface $\Gamma_0$ is given by deep grating profiles $F_\ell(x_1)=2.5\cos{x_1}$ (top panel) and $F_0(x_1)=2.5 \pi(0.4\cos(x_1)-0.2\cos(2x_1)+0.4\cos(3x_1))$ (bottom panel), and various values of wavenumbers $k_\ell, \ell=0,1$. \label{comp1a}}
\end{center}
  \end{table}

  \begin{table}
   \begin{center}
     \resizebox{!}{1.1cm}
{   
\begin{tabular}{|c|c|c|c|}
  \hline
  \multicolumn{1}{|c|} {$k_0=1.3, k_1=4.3$}& \multicolumn{1}{c|} {$k_0=2.3, k_1=8.3$}& \multicolumn{1}{c|} {$k_0=4.3, k_1=16.3$}& \multicolumn{1}{c|} {$k_0=16.3, k_1=64.3$}\\
\cline{1-4}
$\mathcal{A}_n,L=0$ & $\mathcal{A}_n, L=0$ & $\mathcal{A}_n, L=0$ & $\mathcal{A}_n, L=0$ \\
\hline
15 & 16 & 18 & 20 \\
\hline
\hline
17 & 19 & 21 & 25\\
\hline
\end{tabular}
}
\caption{Numbers of GMRES iterations required by QO DD formulations to reach relative residuals of $10^{-6}$ for configurations consisting of $2$ layers, where the interface $\Gamma_0$ is given by the Lipschitz grating profiles depicted in Figure~\ref{fig:Lipschitz} with height/period =1 (top panel corresponds to the triangle grating and the bottom panel corresponds to the lamellar grating), and various values of wavenumbers $k_\ell, \ell=0,1$. \label{comp1aa}}
\end{center}
  \end{table}

  \emph{Three layers.} We devote the next set of results to configurations consisting of three layers separated by two periodic interfaces. We present in Table~\ref{comp7aa} numbers of GMRES iterations required by the QO DD discretization matrices $\mathcal{A}_n$ to reach relative residuals of $10^{-4}$ for increasingly rougher (yet smooth) grating profiles separating high-contrast periodic layers. We remark that for small values of the roughness parameter $\varepsilon$ (i.e. $\varepsilon=0.1,\ 0.5$), the numbers of iterations do not appear to depend on the increased contrast.  For larger values of the parameter $\varepsilon$ (i.e. $\varepsilon=1$), the numbers of iterations grow with the frequency, yet the growth rate is modest. We also point out that the use of transmission operators $Z_{j,j+1}^2$ (which are higher-order approximations of the DtN operators) appears to be beneficial to the iterative behavior of the QO DD algorithm.
  \begin{table}
   \begin{center}
     \resizebox{!}{1.1cm}
{   
\begin{tabular}{|c|c|c|c|c|c|c|c|c|c|c|c|c|}
  \hline
  $\varepsilon$ & \multicolumn{4}{c|} {$k_0=1.3, k_1=4.3, k_2=16.3$}& \multicolumn{4}{c|} {$k_0=2.3, k_1=8.3, k_2 =32.3$}& \multicolumn{4}{c|} {$k_0=4.3, k_1=16.3, k_3=64.3$}\\
\cline{1-13}
& $\mathcal{A}_n,L=0$ & $\mathcal{B}_n^{-1}\mathcal{A}_n,L=0$ & $\mathcal{A}_n,L=2$ & $\mathcal{B}_n^{-1}\mathcal{A}_n,L=2$ & $\mathcal{A}_n,L=0$ & $\mathcal{B}_n^{-1}\mathcal{A}_n,L=0$ & $\mathcal{A}_n,L=2$ & $\mathcal{B}_n^{-1}\mathcal{A}_n,L=2$& $\mathcal{A}_n,L=0$ & $\mathcal{B}_n^{-1}\mathcal{A}_n,L=0$ & $\mathcal{A}_n,L=2$ & $\mathcal{B}_n^{-1}\mathcal{A}_n,L=2$ \\
\hline
 0.1 & 9 & 9 & 9 & 9 & 9 & 9 & 9 & 9 & 9 & 9 & 9 & 9\\
\hline
0.5 & 13 & 11 & 11 & 11 & 14 & 12 & 11 & 11 & 14 & 12 & 11 & 11\\
\hline
1 & 20 & 17 & 16 & 14 & 23 & 19 & 16  & 14 & 27 & 23 & 19 & 16\\
\hline
\hline
0.1 & 10 & 9 & 10 & 9 & 11 & 11 & 11 & 11 & 10 & 9 & 9 & 9\\
\hline
0.5 & 18 & 15 & 13 & 11 & 18 & 15 & 14 & 13 & 19 & 15 & 17 & 14\\
\hline
1 & 30 & 25 & 20 & 19 & 38 & 33 & 22 & 21 & 56 & 49 & 24 & 22\\
\hline
\end{tabular}
}
\caption{Numbers of GMRES iterations required by QO DD formulations to reach relative residuals of $10^{-4}$ for configurations consisting of $3$ layers, where the interfaces $\Gamma_\ell,0\leq \ell\leq 1$ are given by grating profiles $F_\ell(x_1)=-\ell H+2.5\varepsilon\cos{x_1}, H=3.3, 0\leq \ell\leq 1$ (top panel) and $F_\ell(x_1)=-\ell H+2.5\pi\varepsilon(0.4\cos(x_1)-0.2\cos(2x_1)+0.4\cos(3x_1)), H=3.3, 0\leq \ell\leq 1$ (bottom panel) for various values of the height parameter $\varepsilon$, under normal incidence, and various values of wavenumbers $k_\ell$. \label{comp7aa}}
\end{center}
  \end{table}

  \begin{figure}
\centering
\includegraphics[scale=0.7]{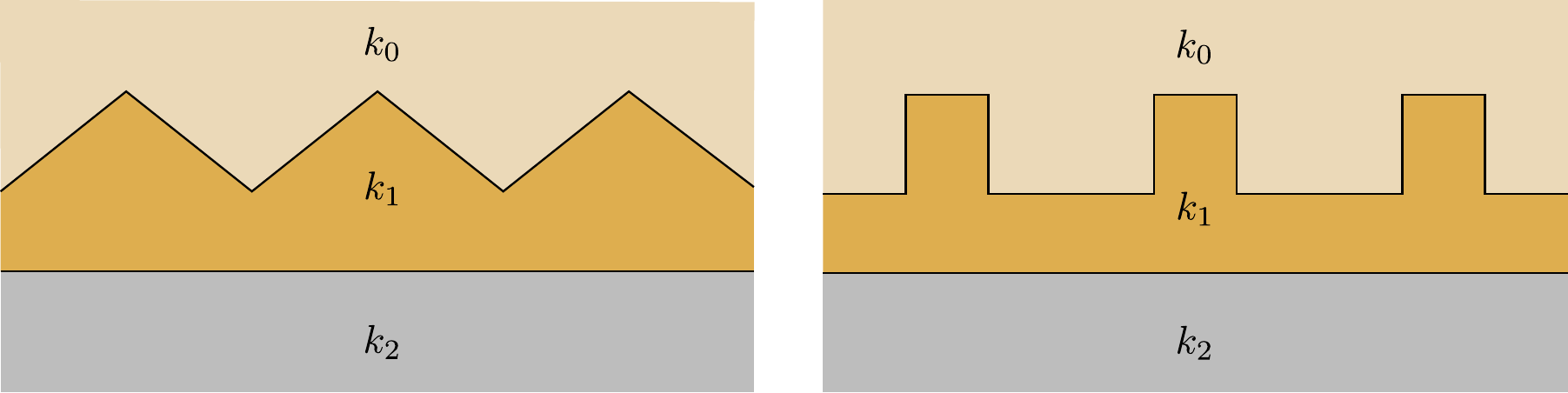}
\caption{Three layer configurations with Lipschitz upper layer grating profiles of period $2\pi$ and height $\varepsilon$.}
\label{fig:Lipschitz_three_layers}
  \end{figure}

  \begin{table}
   \begin{center}
     \resizebox{!}{1.1cm}
{   
\begin{tabular}{|c|c|c|c|c|c|c|}
  \hline
  $\varepsilon$ & \multicolumn{2}{c|} {$k_0=1.3, k_1=4.3, k_2=16.3$}& \multicolumn{2}{c|} {$k_0=2.3, k_1=8.3, k_2 =32.3$}& \multicolumn{2}{c|} {$k_0=4.3, k_1=16.3, k_3=64.3$}\\
\cline{1-7}
& $\mathcal{A}_n,L=0$ & $\mathcal{B}_n^{-1}\mathcal{A}_n,L=0$ & $\mathcal{A}_n,L=0$ & $\mathcal{B}_n^{-1}\mathcal{A}_n,L=0$ & $\mathcal{A}_n,L=0$ & $\mathcal{B}_n^{-1}\mathcal{A}_n,L=0$ \\
\hline
 0.25 & 15 & 13 & 18 & 15 & 22 & 18\\
\hline
1.25 & 16 & 14 & 18 & 16 & 21 & 18\\
\hline
2.5 & 18 & 16 & 19 & 17 & 21 & 19\\
\hline
\hline
0.25 & 18 & 17 & 21 & 18 & 25 & 23\\
\hline
1.25 & 18 & 17 & 21 & 18 & 26 & 23\\
\hline
2.5 & 18 & 17 & 25 & 22 & 28 & 25 \\
\hline
\end{tabular}
}
\caption{Numbers of GMRES iterations required by QO DD formulations to reach relative residuals of $10^{-4}$ for configurations consisting of $3$ layers as depicted in Figure~\ref{fig:Lipschitz_three_layers}. The smallest distance between the grating profiles $\Gamma_0$ and $\Gamma_1$ was taken to be equal to $1.3$. The results corresponding to the top triangular grating are presented in the top panel, and the results corresponding to the top lamellar grating are presented in the bottom panel. We considered various values of the height parameter $\varepsilon$ for the top grating profiles. for various values of the height parameter $\varepsilon$, under normal incidence, and various values of wavenumbers $k_\ell$.  \label{comp7aaa}}
\end{center}
  \end{table}
 
  \emph{Many layers.} We investigate next the iterative behavior of the QO DD solvers and the effectiveness of the sweeping preconditioners in the case of configurations that involve large numbers of layers. In the case when the roughness of the interfaces is small enough (i.e. the height parameter $\varepsilon=0.02$), we see in Table~\ref{comp8aa} that the sweeping preconditioner applied to the QO DD matrices $\mathcal{A}_n$ appears to be scalable, that is the numbers of GMRES iterations required for convergence does not depend on the number of layers or on the frequencies in each layer. We note that the transmission problems considered in Table~\ref{comp8aa} (as well as in Table~\ref{comp8a}-Table~\ref{comp8ab}) range from 100 to 4,000 wavelengths---as measured by the number of wavelengths across the period of each interface; the discretization size ranges from 5,000 to 15,000 unknowns. As the roughness parameter is increased (that is $\varepsilon=0.1$), we see in Table~\ref{comp8a} that the numbers of GMRES iterations do still scale when the sweeping preconditioner is applied to the DD matrices $\mathcal{A}_n^s$, but not in its counterpart case involving the DD matrices $\mathcal{A}_n$. Furthermore, QO DD solvers based on higher-order transmission operators (that is values of the parameter $L\geq 1$ in the definition of the transmission operators $Z_{j,j-1}^L$,$Z_{j,j+1}^L$, and respectively $Z_{j,j-1}^{s,L}$ and $Z_{j,j+1}^{s,L}$) perform only marginally better than those based on first-order transmission operators (that is $L=0$ in the definition of the aforementioned transmission operators) in the case of small roughness parameters $\varepsilon$. Based on our numerical experience, we observed that the iterative behavior of the QO DD solvers and the sweeping preconditioners depicted in Table~\ref{comp8aa} and Table~\ref{comp8a} is not sensitive to the width $H$ of the layers or the shape of the grating profiles. Furthermore, qualitatively similar behavior was observed in the cases when the interfaces $\Gamma_\ell$ are Lipschitz.
  
   \begin{table}
   \begin{center}
     \resizebox{!}{1.1cm}
{   
\begin{tabular}{|c|c|c|c|c|c|c|c|c|c|c|c|c|}
  \hline
  $N$ & \multicolumn{4}{c|} {$k_\ell=\ell+1.3, 0\leq \ell\leq N$}& \multicolumn{4}{c|} {$k_\ell=2\ell+1.3,0\leq\ell\leq N$}& \multicolumn{4}{c|} {$k_\ell=4\ell+1.3, 0\leq\ell\leq N$}\\
\cline{1-13}
& $\mathcal{A}_n,L=0$ & $\mathcal{B}_n^{-1}\mathcal{A}_n,L=0$ & $\mathcal{A}_n,L=2$ & $\mathcal{B}_n^{-1}\mathcal{A}_n,L=2$ & $\mathcal{A}_n,L=0$ & $\mathcal{B}_n^{-1}\mathcal{A}_n,L=0$ & $\mathcal{A}_n,L=2$ & $\mathcal{B}_n^{-1}\mathcal{A}_n,L=2$& $\mathcal{A}_n,L=0$ & $\mathcal{B}_n^{-1}\mathcal{A}_n,L=0$ & $\mathcal{A}_n,L=2$ & $\mathcal{B}_n^{-1}\mathcal{A}_n,L=2$ \\
\hline
9 & 58 & 13 & 58 & 13 & 62 & 13 & 62 & 13 & 61 & 12 & 61 & 12\\
\hline
19 & 115 & 14 & 115 & 14 & 119 & 14 & 119 & 14 & 118 & 14 & 118 & 14\\
\hline
29 & 164 & 14 & 164 & 14 & 171 & 14 & 171  & 14 & 168 & 14 & 168 & 14\\
\hline
\hline
9 & 60 & 13 & 60 & 13 & 58 & 14 & 58 & 14 & 61 & 14 & 61 & 14\\
\hline
19 & 119 & 14 & 115 & 14 & 119 & 16 & 117 & 14 & 117 & 14 & 150 & 14\\
\hline
29 & 167 & 14 &  161 & 14 & 166 & 15 & 164 & 14 & 183 & 14 & 167 & 14\\
\hline
\end{tabular}
}
\caption{Numbers of GMRES iterations required by preconditioned/unpreconditioned QO DD formulations to reach relative residuals of $10^{-4}$ for configurations consisting of $N+2$ layers for various values of $N$, where the interfaces $\Gamma_\ell,0\leq \ell\leq N$ are given by grating profiles $F_\ell(x_1)=-\ell H+2.5\varepsilon\cos{x_1}, H=3.3, 0\leq \ell\leq N$ with $\varepsilon=0.02$ (top panel) and $F_\ell(x_1)=-\ell H+2.5\pi\varepsilon(0.4\cos(x_1)-0.2\cos(2x_1)+0.4\cos(3x_1)), H=3.3, 0\leq \ell\leq N$ with $\varepsilon=0.02$ (bottom panel), under normal incidence, and various values of wavenumbers $k_\ell$. \label{comp8aa}}
\end{center}
  \end{table}

   \begin{table}
   \begin{center}
     \resizebox{!}{1.1cm}
{   
\begin{tabular}{|c|c|c|c|c|c|c|c|c|c|c|c|c|}
  \hline
  $N$ & \multicolumn{4}{c|} {$k_\ell=\ell+1.3, 0\leq \ell\leq N$}& \multicolumn{4}{c|} {$k_\ell=2\ell+1.3,0\leq\ell\leq N$}& \multicolumn{4}{c|} {$k_\ell=4\ell+1.3, 0\leq\ell\leq N$}\\
\cline{1-13}
& $\mathcal{A}_n,L=2$ & $\mathcal{B}_n^{-1}\mathcal{A}_n,L=2$ & $\mathcal{A}_n^s,L=2$ & $(\mathcal{B}_n^s)^{-1}\mathcal{A}_n^s,L=2$ & $\mathcal{A}_n,L=2$ & $\mathcal{B}_n^{-1}\mathcal{A}_n,L=2$ & $\mathcal{A}_n^s,L=2$ & $(\mathcal{B}_n^s)^{-1}\mathcal{A}_n^s,L=2$& $\mathcal{A}_n,L=2$ & $\mathcal{B}_n^{-1}\mathcal{A}_n,L=2$ & $\mathcal{A}_n^s,L=2$ & $(\mathcal{B}_n^s)^{-1}\mathcal{A}_n^s,L=2$ \\
\hline
9 & 60 & 18 & 60 & 15 & 58 & 17 & 58 & 14 & 57 & 15 & 57 & 14\\
\hline
19 & 109 & 20 & 109 & 18 & 110 & 20 & 110 & 15 & 104 & 19 & 104 & 14\\
\hline
29 & 156 & 22 & 156 & 18 & 161 & 25 & 159 & 16 & 157 & 24 & 154 & 14\\
\hline
\hline
9 & 64 & 18 & 61 & 14 & 59 & 19 & 59 & 15 & 60 & 20 & 55 & 14\\
\hline
19 & 117 & 26 & 111 & 16 & 119 & 30 & 108 & 16 & 125 & 33 & 107 & 14\\
\hline
29 & 170 & 36 & 156 & 16 & 179 & 48 & 155 & 15 & 188 & 54 & 152 & 14\\
\hline
\end{tabular}
}
\caption{Numbers of GMRES iterations required by various preconditioned/unpreconditioned QO DD formulations to reach relative residuals of $10^{-4}$ for configurations consisting of $N+2$ layers for various values of $N$, where the interfaces $\Gamma_\ell,0\leq \ell\leq N$ are given by grating profiles $F_\ell(x_1)=-\ell H+2.5\varepsilon\cos{x_1}, H=3.3, 0\leq \ell\leq N$ with $\varepsilon=0.1$ (top panel) and $F_\ell(x_1)=-\ell H+2.5\pi\varepsilon(0.4\cos(x_1)-0.2\cos(2x_1)+0.4\cos(3x_1)), H=3.3, 0\leq \ell\leq N$ with $\varepsilon=0.1$ (bottom panel), under normal incidence, and various values of wavenumbers $k_\ell$. \label{comp8a}}
\end{center}
   \end{table}

   As the roughness of the gratings $\Gamma_\ell$ increases, the sweeping preconditioner $(\mathcal{B}_n^s)^{-1}\mathcal{A}_n^s$ is still effective, yet the number of iterations required by the preconditioned formulation $(\mathcal{B}_n^s)^{-1}\mathcal{A}_n^s$ grows mildly with the number of layers as well as with increased frequencies/contrasts---see Table~\ref{comp8aba}. Remarkably, there are important benefits in the reduction of GMRES iterations by incorporating higher-order transmission operators $Z^{s,2}_{j,j-1}$ and $Z^{s,2}_{j,j+1}$ over the first-order ones $Z^{s,0}_{j,j-1}$ and $Z^{s,0}_{j,j+1}$ in the preconditioned QO DD formulations. Also, the sweeping preconditioner is less effective for QO DD formulations based on transmission operators $Z^2_{j,j-1}$ and $Z^2_{j,j+1}$ (i.e. $\mathcal{B}_n^{-1}\mathcal{A}_n$) for rough interface profiles. This behavior can be attributed to ill-conditioning present in the recursions~\eqref{eq:recursiveYj} for rough profiles---see Remark~\ref{stab_rec}. We note that the recursions~\eqref{eq:recursive_1} that are the basis of the evaluation of the transmission operators $Z^{s,L}_{j,j-1}$ and $Z^{S,L}_{j,j+1}$ are stable.  Finally, it can be seen from the results presented in Table~\ref{comp8aba} that QO DD solvers based on the formulation $(\mathcal{B}_n^\flat)^{-1}\mathcal{A}_n^\flat$ (which, given that the depth of the layers is larger than the profile roughness, is applicable in the case presented in Table~\ref{comp8aba}) require small numbers of GMRES iterations for convergence, whose growth with respect to the number of layers or contrast is very mild.

      \begin{table}
   \begin{center}
     \resizebox{!}{1.0cm}
{   
\begin{tabular}{|c|c|c|c|c|c|c|c|c|c|c|c|c|}
  \hline
  $N$ & \multicolumn{4}{c|} {$k_\ell=\ell+1.3, 0\leq \ell\leq N$}& \multicolumn{4}{c|} {$k_\ell=2\ell+1.3,0\leq\ell\leq N$}& \multicolumn{4}{c|} {$k_\ell=4\ell+1.3, 0\leq\ell\leq N$}\\
\cline{1-13}
& $\mathcal{A}_n,L=0$ & $\mathcal{B}_n^{-1}\mathcal{A}_n,L=0,2$ & $(\mathcal{B}_n^s)^{-1}\mathcal{A}_n^s,L=0,2$ & $(\mathcal{B}_n^\flat)^{-1}\mathcal{A}_n^\flat$ & $\mathcal{A}_n,L=0$ & $\mathcal{B}_n^{-1}\mathcal{A}_n,L=0,2$ & $(\mathcal{B}_n^s)^{-1}\mathcal{A}_n^s,L=0,2$& $(\mathcal{B}_n^\flat)^{-1}\mathcal{A}_n^\flat$ & $\mathcal{A}_n,L=0$ & $\mathcal{B}_n^{-1}\mathcal{A}_n,L=0,2$ & $(\mathcal{B}_n^s)^{-1}\mathcal{A}_n^s,L=0,2$ & $(\mathcal{B}_n^\flat)^{-1}\mathcal{A}_n^\flat$ \\
\hline
9 & 88 & 44/41 & 44/25 & 32 & 94 & 53/48 & 53/25 & 29 & 100 & 61/55 & 61/25 & 26\\
\hline
19 & 199 & 101/98 & 100/34 & 36 & 224 & 147/134 & 145/39 & 30 & 247 & 156/133 & 123/47 & 28\\
\hline
29 & 324 & 152/144 & 149/40 & 41 & 359 & 273/260 & 261/52 & 30 & 387 & 270/215 & 254/63 & 30\\
\hline
\hline
9 & 121 & 58/57 & 57/28 & 23 & 129 & 82/66 & 84/30 & 28 & 141 & 71/67 & 79/34 & 30 \\
\hline
19 & 253 & 130/120 & 124/48 & 30 & 317 & 221/182 & 218/60 & 32 & 388 & 187/160 & 162/78 & 34 \\
\hline
29 & 390 & 206/177 & 190/68 & 32 & 502 & 296/233 & 251/75 & 34 & 862 & 389/354 & 375/86 & 37 \\
\hline
\end{tabular}
}
\caption{Numbers of GMRES iterations required by various preconditioned/unpreconditioned QO DD formulations to reach relative residuals of $10^{-4}$ for configurations consisting of $N+2$ layers for various values of $N$, where the interfaces $\Gamma_\ell,0\leq \ell\leq N$ are given by grating profiles $F_\ell(x_1)=-\ell H+2.5\varepsilon\cos{x_1}, H=3.3, 0\leq \ell\leq N$ with $\varepsilon=0.5$ (top panel) and $F_\ell(x_1)=-\ell H+2.5\pi\varepsilon(0.4\cos(x_1)-0.2\cos(2x_1)+0.4\cos(3x_1)), H=3.3, 0\leq \ell\leq N$ with $\varepsilon=0.5$ (bottom panel), under normal incidence, and various values of wavenumbers $k_\ell$. \label{comp8aba}}
\end{center}
   \end{table}

   In the case of very rough gratings $\Gamma_\ell$ (whose height/period ratios are close to $1$), the sweeping preconditioners $(\mathcal{B}_n^s)^{-1}\mathcal{A}_n^s$ become less effective--see Table~\ref{comp8ab}. This can be attributed to the fact that the large values of the roughness parameter $\varepsilon$ are outside the radius of convergence of the perturbation series approximations of DtN operators. Nevertheless, the use of higher-order transmission operators $Z^{s,2}_{j,j-1}$ and $Z^{s,2}_{j,j+1}$ is again beneficial. We mention that due the ratio between the profile roughness and the width of the layers, the DD formulation~\eqref{DDM_t_final} is not possible in this case: a strip domain decomposition would necessarily require that the flat subdomain interfaces intersect the gratings $\Gamma_\ell$. Nevertheless, once the layers width is large enough with respect to profile roughness so that the DD formulation~\eqref{DDM_t_final} is possible, the sweeping preconditioners $(\mathcal{B}_n^\flat)^{-1}\mathcal{A}_n^\flat$ are effective---see Table~\ref{comp8abc}.

\begin{table}
   \begin{center}
     \resizebox{!}{0.8cm}
{   
\begin{tabular}{|c|c|c|c|c|c|c|c|c|c|c|c|c|}
  \hline
  $N$ & \multicolumn{4}{c|} {$k_\ell=\ell+1.3, 0\leq \ell\leq N$}& \multicolumn{4}{c|} {$k_\ell=2\ell+1.3,0\leq\ell\leq N$}& \multicolumn{4}{c|} {$k_\ell=4\ell+1.3, 0\leq\ell\leq N$}\\
\cline{1-13}
& $\mathcal{A}_n^s,L=0$ & $(\mathcal{B}_n^s)^{-1}\mathcal{A}_n^s,L=0$ &  $\mathcal{A}_n^s,L=2$ & $(\mathcal{B}_n^s)^{-1}\mathcal{A}_n^s,L=2$ & $\mathcal{A}_n^s,L=0$ & $(\mathcal{B}_n^s)^{-1}\mathcal{A}_n^s,L=0$ &  $\mathcal{A}_n^s,L=2$ & $(\mathcal{B}_n^s)^{-1}\mathcal{A}_n^s,L=2$ & $\mathcal{A}_n^s,L=0$ & $(\mathcal{B}_n^s)^{-1}\mathcal{A}_n^s,L=0$ & $\mathcal{A}_n^s,L=2$ & $(\mathcal{B}_n^s)^{-1}\mathcal{A}_n^s,L=2$\\
\hline
9 & 195 & 137 & 83 & 57 & 260 & 191 & 92 & 66 & 310 & 280 & 135 & 87\\
\hline
19 & 522 & 312 & 171 & 109 & 696 & 524 & 253 & 204 & 1080 & 812 & 390 & 281\\
\hline
\hline
9 & 266 & 164 & 103 & 76 & 390 & 254 & 121 & 84 & 481 & 416 & 166 & 125\\
\hline
19 & 736 & 392 & 256 & 187 & 1145 & 658 & 354 & 247 & 1801 & 1223 & 461 & 312\\
\hline

\end{tabular}
}
\caption{Numbers of GMRES iterations required by various preconditioned/unpreconditioned QO DD formulations to reach relative residuals of $10^{-4}$ for configurations consisting of $N+2$ layers for various values of $N$, where the interfaces $\Gamma_\ell,0\leq \ell\leq N$ are given by grating profiles $F_\ell(x_1)=-\ell H+2.5\varepsilon\cos{x_1}, H=3.3, 0\leq \ell\leq N$ with $\varepsilon=1$ (top panel) and $F_\ell(x_1)=-\ell H+2.5\pi\varepsilon(0.4\cos(x_1)-0.2\cos(2x_1)+0.4\cos(3x_1)), H=3.3, 0\leq \ell\leq N$ with $\varepsilon=1$ (bottom panel), under normal incidence, and various values of wavenumbers $k_\ell$.\label{comp8ab}}
\end{center}
\end{table}

   \begin{table}
   \begin{center}
     \resizebox{!}{1.0cm}
{   
\begin{tabular}{|c|c|c|c|c|c|c|}
  \hline
  $N$ & \multicolumn{2}{c|} {$k_\ell=\ell+1.3, 0\leq \ell\leq N$}& \multicolumn{2}{c|} {$k_\ell=2\ell+1.3,0\leq\ell\leq N$}& \multicolumn{2}{c|} {$k_\ell=4\ell+1.3, 0\leq\ell\leq N$}\\
\cline{1-7}
& $(\mathcal{B}_n^s)^{-1}\mathcal{A}_n^s,L=2$ & $(\mathcal{B}_n^\flat)^{-1}\mathcal{A}_n^\flat$ & $)\mathcal{B}_n^s)^{-1}\mathcal{A}_n^s,L=2$ & $(\mathcal{B}_n^\flat)^{-1}\mathcal{A}_n^\flat$ &$(\mathcal{B}_n^s)^{-1}\mathcal{A}_n^s,L=2$ & $(\mathcal{B}_n^\flat)^{-1}\mathcal{A}_n^\flat$ \\
\hline
9 & 50 & 27 & 66 & 26 & 110 & 26\\
\hline
19 & 90 & 27 & 158 & 28 & 189 & 29\\
\hline
\hline
9 & 73 & 26 & 98 & 28 & 150 & 29\\
\hline
19 & 139 & 31 & 183 & 34 & 247 & 35\\
\hline
\end{tabular}
}
\caption{Numbers of GMRES iterations required by various preconditioned QO DD formulations to reach relative residuals of $10^{-4}$ for configurations consisting of $N+2$ layers for various values of $N$, where the interfaces $\Gamma_\ell,0\leq \ell\leq N$ are given by grating profiles $F_\ell(x_1)=-\ell H+2.5\cos{x_1}, H=5.6, 0\leq \ell\leq N$ (top panel) and $F_\ell(x_1)=-\ell H+2.5\pi(0.4\cos(x_1)-0.2\cos(2x_1)+0.4\cos(3x_1)), H=4.5, 0\leq \ell\leq N$ (bottom panel), under normal incidence, and various values of wavenumbers $k_\ell$.\label{comp8abc}}
\end{center}
   \end{table}

   According to the results presented in Tables~\ref{comp8aa}--\ref{comp8abc}, the sweeping preconditioners $\mathcal{B}_n^{-1}\mathcal{A}_n$ and especially $(\mathcal{B}_n^s)^{-1}\mathcal{A}_n^s$ can effectively reduce the numbers of GMRES iterations required for the solution of QO DD algorithms for periodic transmission problems involving large number of layers, even in the case when the roughness of the interfaces of material discontinuity is pronounced. We have observed that these findings are virtually independent of the layer material properties (for instance, the numbers of GMRES iterations reported in these tables are about the same when we considered random wavenumbers in the same range) or the depth of the layers (as long as the original transmission problem is well posed). In addition, the sweeping preconditioners $(\mathcal{B}_n^\flat)^{-1}\mathcal{A}_n^\flat$, whenever applicable, are extremely efficient, even for very rough profiles $\Gamma_\ell$. As we have presented in Tables~\ref{comp8aa}--\ref{comp8abc}, the choice of the transmission operators plays an important role in the convergence properties of the ensuing DD algorithms. Besides the square root Fourier multiplier transmission operators presented in this paper, other transmission operators have been used in the DD arena. Notably, me mention the classical Robin transmission operators $Z=i\ Id$ (the first transmission operators introduced for DD formulations of Helmholtz equations by D\'espres ~\cite{Depres}), as well as transmission operators of the form $Z=i\mathcal{T}$ (related to the ones introduced in~\cite{lecouvez2014quasi}), where the operator $\mathcal{T}$ is related to the Hilbert transform
   \begin{equation}
     \mathcal{T}(\varphi)(t)=\partial_t\int_0^{2\pi}\mathcal{K}(t-\tau)\partial_{\tau}\varphi(\tau)d\tau+\varphi(t),\quad \mathcal{K}(t):=\frac{1}{\pi}\ln{|1-e^{it}|},\quad 0\leq t\leq 2\pi,
   \end{equation}
   where $\varphi$ is a $2\pi$ periodic function. We note that these two choices of transmission operators give rise to isometric RtR maps, and thus they lead to DD formulations that are well-posed as long as the initial transmission problem~\eqref{system_t} is. We illustrate in Figure~\ref{fig:iters} the numbers of iterations required by DD formulations that rely on the two above mentioned transmission operators. Specifically, we considered profiles defined by $F_\ell(x_1)=-\ell H+2.5\varepsilon\cos{x_1}, H=3.3, 0\leq \ell\leq N$ with $\varepsilon=0.1$ and we report numbers of GMRES iterations required by the DD with the transmission operators defined above to reach relative residuals of $10^{-4}$.  Comparing the results in Figure~\ref{fig:iters} with their counterparts in Table~\ref{comp8a}, we see that the use of DD with QO transmission operators $Z^{s,L}_{j,j-1}$ and $Z^L_{j,j+1}$ in conjunction with sweeping preconditioners can give rise to order of magnitude reductions in numbers of GMRES iterations. We mention that the sweeping preconditioner is ineffective in the cases presented in Figure~\ref{fig:iters}. This finding is not surprising, given that the premise of sweeping preconditioners is that the transmission operators are good approximations of subdomain DtN maps. Finally, similar scenarios occur for rougher profiles. 
   \begin{figure}
\centering
\includegraphics[height=80mm]{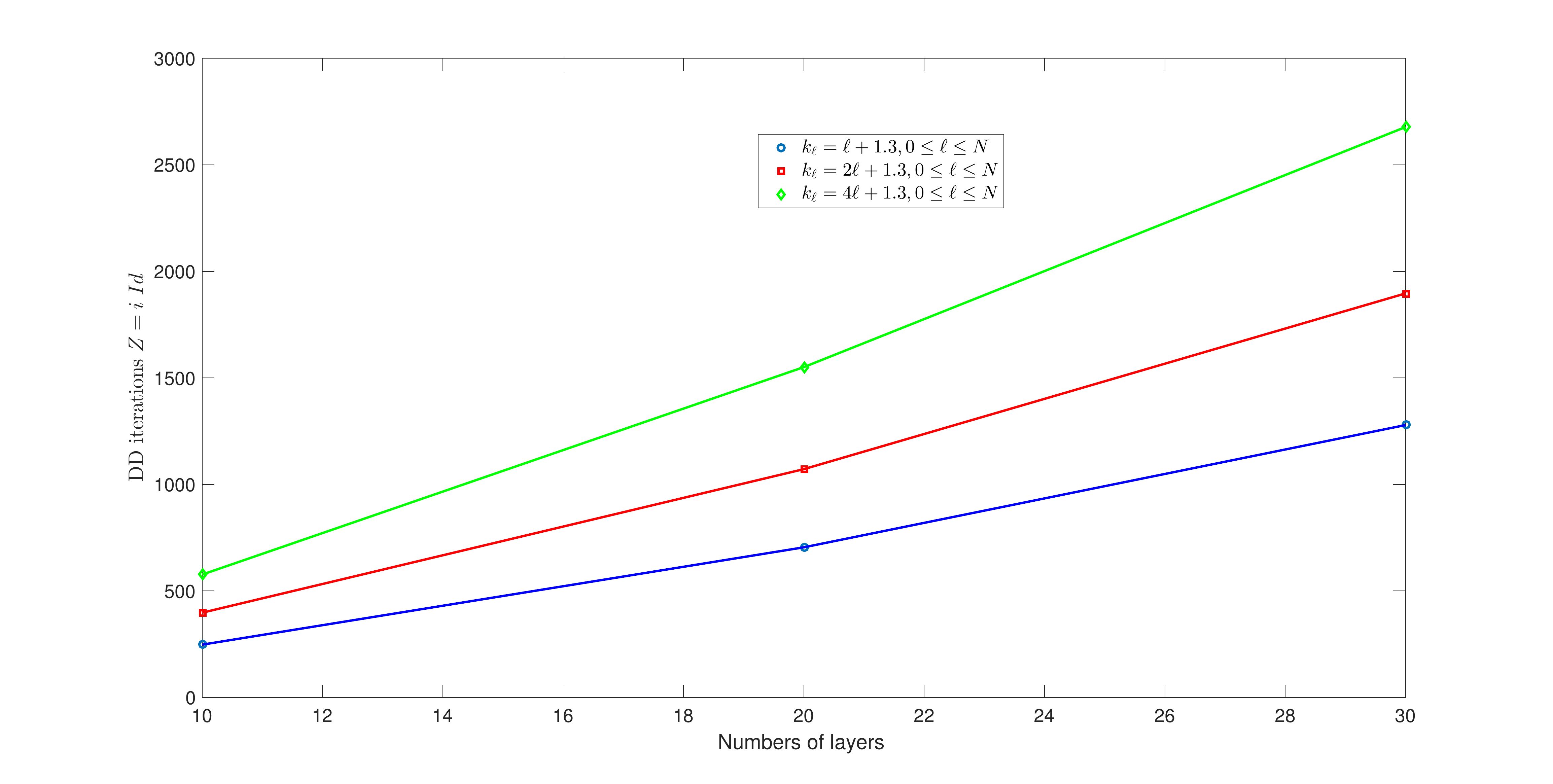}\\
\includegraphics[height=80mm]{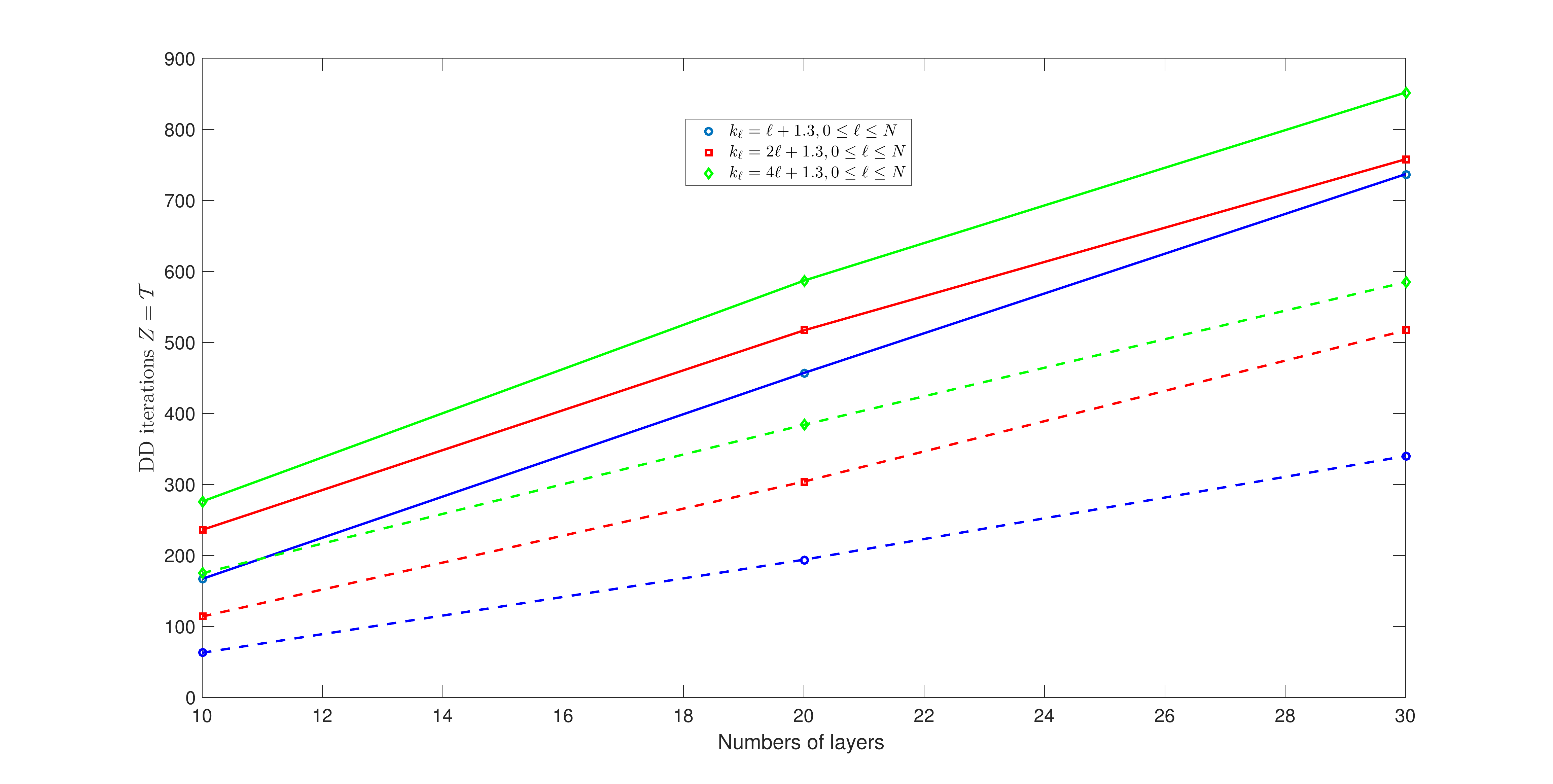}
\caption{Numbers of GMRES iterations required to reach relative residuals of $10^{-4}$ by DD algorithms based on transmission operators $Z=i I$ (top) and $Z=i\mathcal{T}$ (bottom) in the case when the profiles are given by the gratings $F_\ell(x_1)=-\ell H+2.5\varepsilon\cos{x_1}, H=3.3, 0\leq \ell\leq N$ with $\varepsilon=0.1$. In the case of transmission operators $Z=i\mathcal{T}$ we plot with dashed lines the numbers of iterations required after the sweeping preconditioner is applied.}
\label{fig:iters}
   \end{figure}
   
Further insight on the superior performance of the DD algorithms based on QO transmission operators $Z^L_{j,j+1}$ can be garnered from the eigenvalue distribution depicted in Figure~\ref{fig:eig}. As it can be seen, the eigenvalues corresponding to the DD matrices $\mathcal{A}_n,L=0$ are clustered around one, and the clustering is even more pronounced for the eigenvalues of the preconditioned matrix  $\mathcal{B}_n^{-1}\mathcal{A}_n,L=0$. In contrast, the distribution of the eigenvalues of the DD matrix corresponding to classical Robin transmission operators $Z=i\ I$ is not conducive to fast convergence of GMRES solvers.  
    \begin{figure}
\centering
\includegraphics[height=50mm]{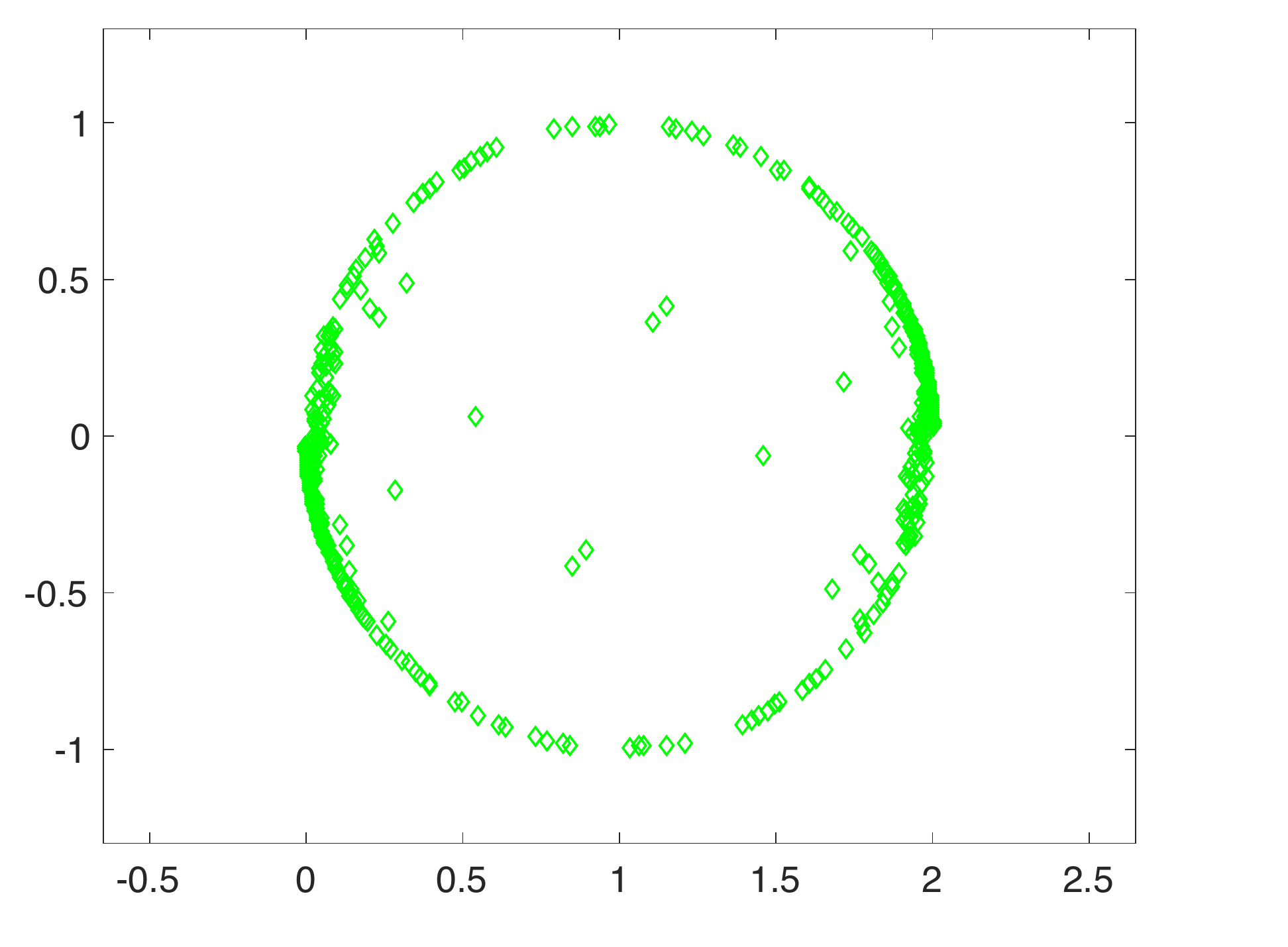}\includegraphics[height=50mm]{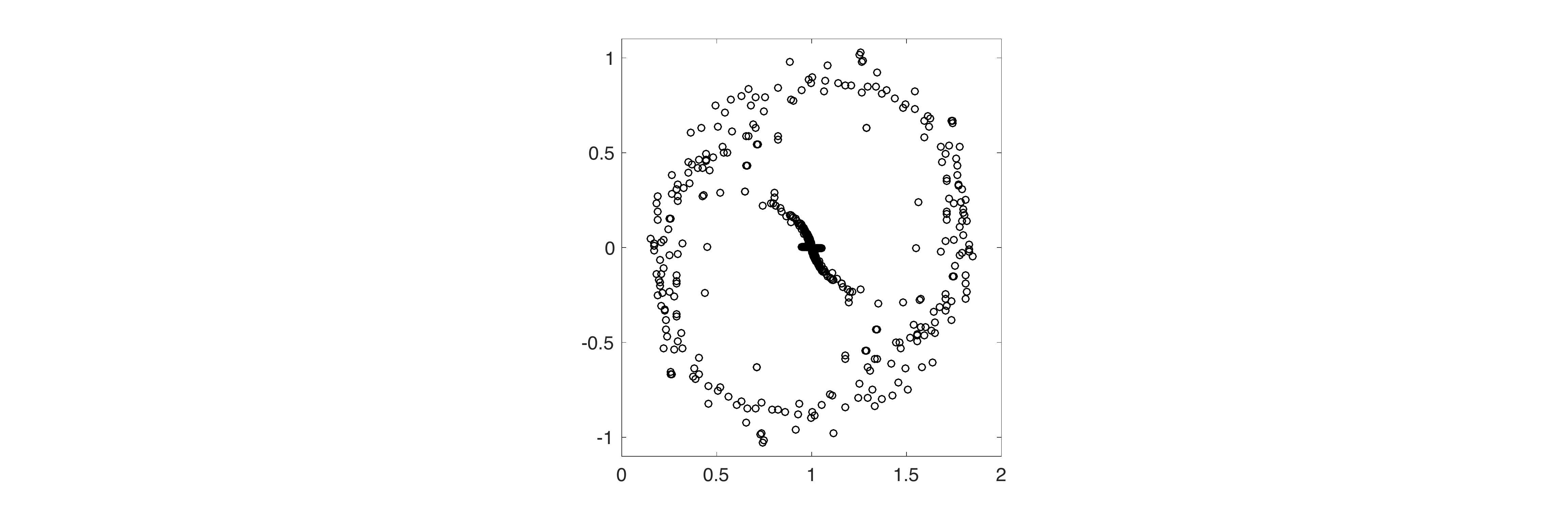}\includegraphics[height=50mm]{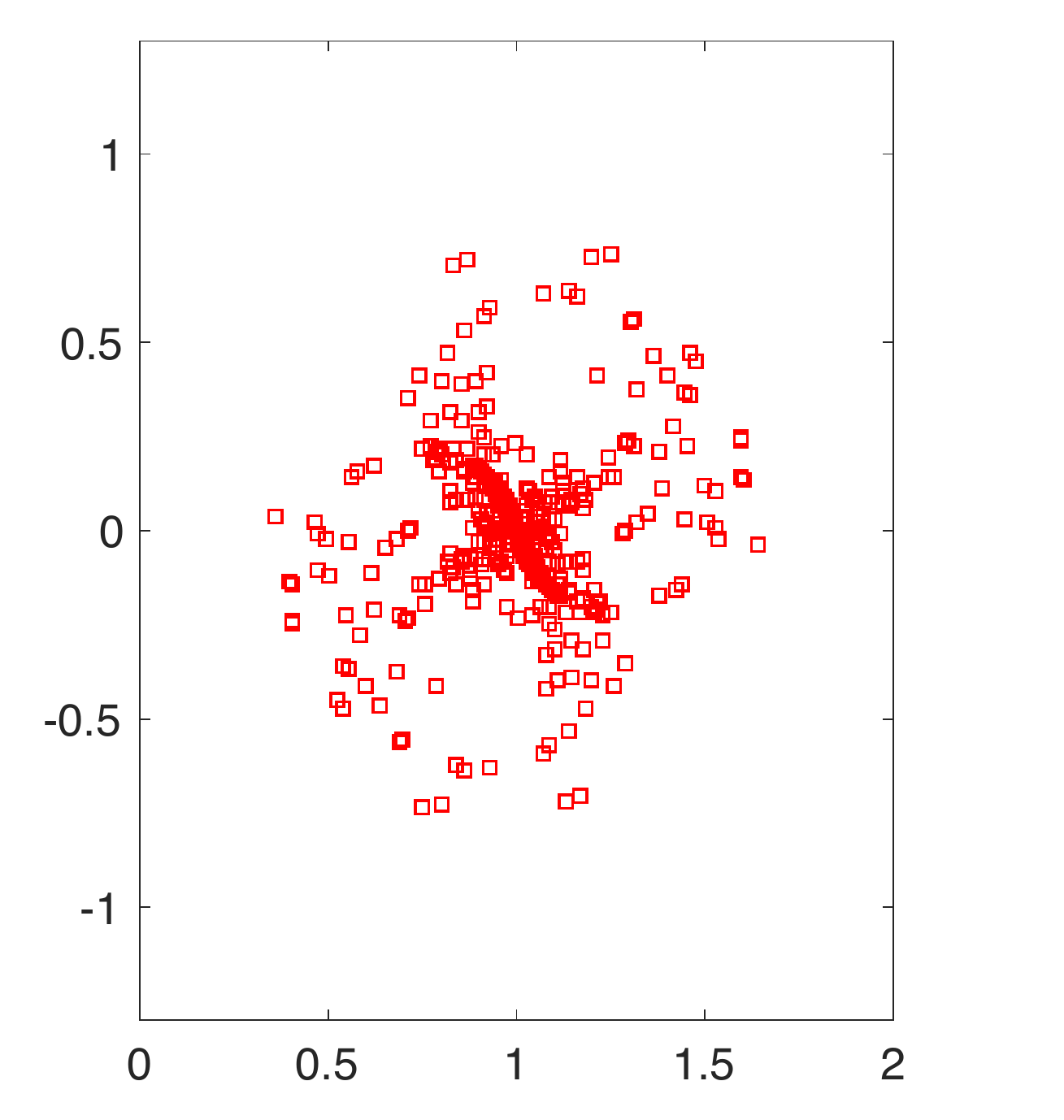}
\caption{Eigenvalue distribution of QO DD formulations for $N=9$, $k_\ell=\ell+1.3, 0\leq \ell\leq N$, $F_\ell(x_1)=-\ell H+2.5\varepsilon\cos{x_1}, H=3.3, 0\leq \ell\leq N$ with $\varepsilon=0.1$: $Z=i\ I$ (left), $\mathcal{A}_n,L=0$ (center), and $\mathcal{B}_n^{-1}\mathcal{A}_n,L=0$ (right).}
\label{fig:eig}
   \end{figure}

   \emph{Inclusions in periodic layers.} Finally, we present results concerning perfectly conducting inclusions embedded in layered media, see Figure~\ref{fig:layers_inclusions}. We present numerical experiments related to such configurations in Table~\ref{comp8aW}. In order to showcase the versatility of our DD algorithm, we chose wavenumbers that are Wood frequencies in the layers that contain inclusions. We note that for these configurations the transmission operators that we use are approximations of DtN operators corresponding to homogeneous layers, and thus the presence of inclusions was not accounted in the construction of transmission operators. Nevertheless, we found that the sweeping preconditioner is still effective, yet the presence of multiple inclusions deteriorates somewhat its performance especially in the high-contrast media cases.  

   \begin{figure}
\centering
\includegraphics[height=60mm]{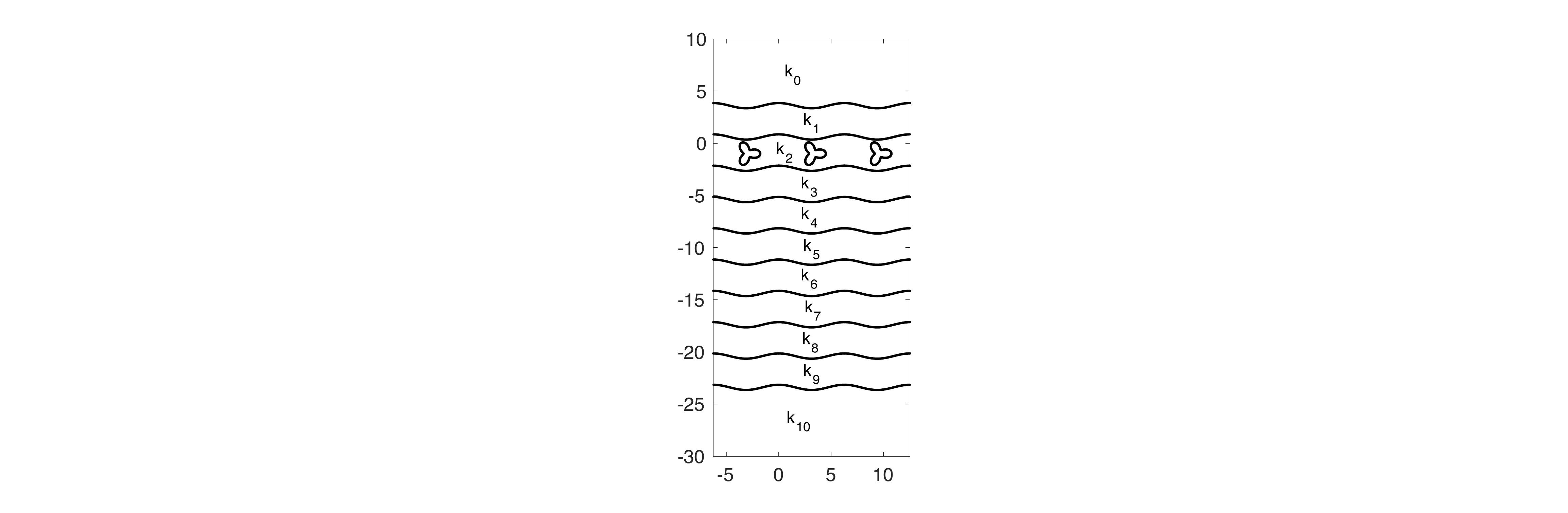}\includegraphics[height=60mm]{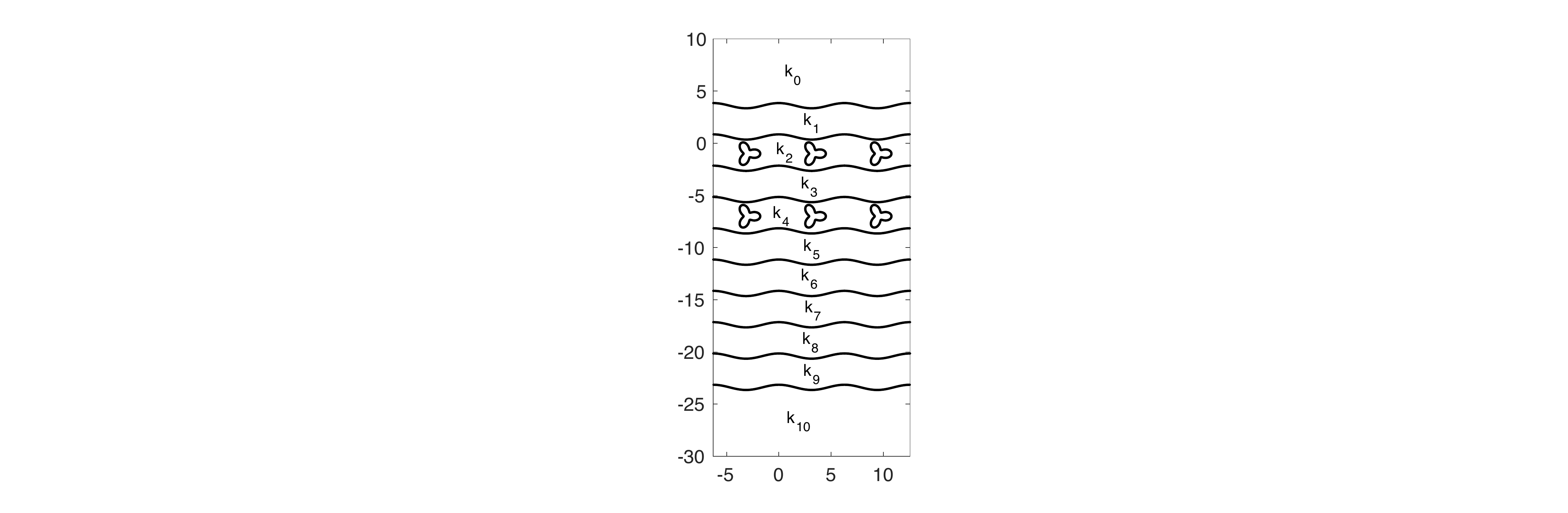}\includegraphics[height=60mm]{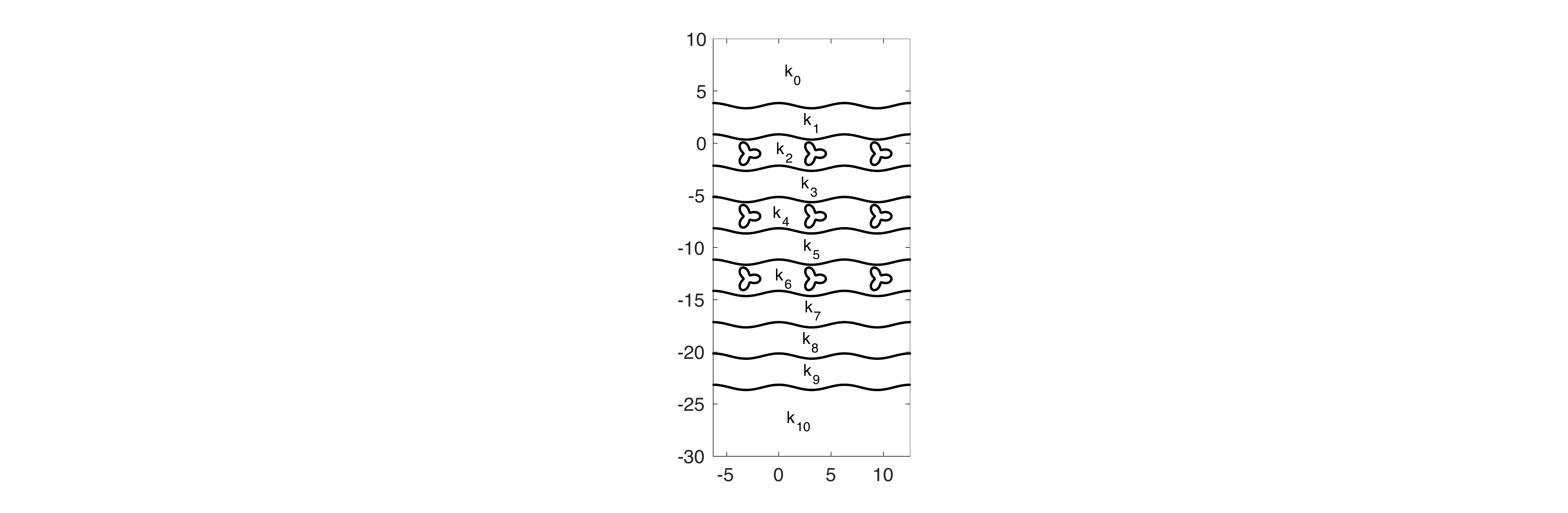}\includegraphics[height=60mm]{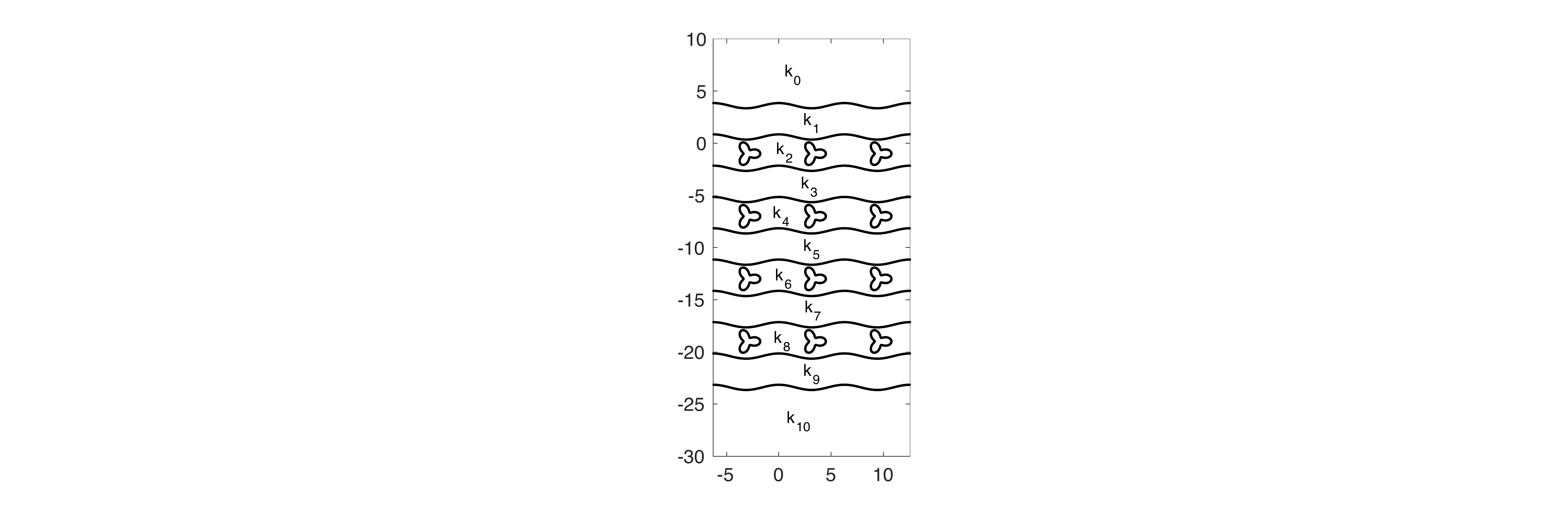}
\caption{Periodic layer configurations with inclusions.}
\label{fig:layers_inclusions}
   \end{figure}

 \begin{table}
   \begin{center}
     \resizebox{!}{1.3cm}
{   
\begin{tabular}{|c|c|c|c|c|c|c|c|c|}
  \hline
  $k_\ell$ & \multicolumn{2}{c|} {One layer inclusions}& \multicolumn{2}{c|} {Two layer inclusions}& \multicolumn{2}{c|} {Three layer inclusions}& \multicolumn{2}{c|} {Four layer inclusions}\\
\cline{1-9}
& $\mathcal{A}_n,L=0$ & $\mathcal{B}_n^{-1}\mathcal{A}_n,L=0$ & $\mathcal{A}_n,L=0$ & $\mathcal{B}_n^{-1}\mathcal{A}_n,L=0$ & $\mathcal{A}_n,L=0$ & $\mathcal{B}_n^{-1}\mathcal{A}_n,L=0$ & $\mathcal{A}_n,L=0$ & $\mathcal{B}_n^{-1}\mathcal{A}_n,L=0$ \\
\hline
$k_\ell=\ell+1.3,\ell\notin I$ & 62 & 17 & 71 & 23 & 81 & 30 & 111& 40 \\
$k_\ell=\ell,\ell\in I$ & & & & &  & & &  \\
\hline
$k_\ell=2\ell+1.3,\ell\notin I$ & 60 & 17 & 69 & 24 & 85 & 31 & 132 & 46 \\
$k_\ell=2\ell,\ell\in I$ & & & & &  & & &  \\
\hline
$k_\ell=4\ell+1.3,\ell\notin I$ & 57 & 18 & 76 & 28 & 89 & 36 & 141 &  49\\
$k_\ell=4\ell,\ell\in I$ & & & & &  & & &  \\
\hline
\end{tabular}
}
\caption{Numbers of GMRES iterations required by various QO DD formulations to reach relative residuals of $10^{-4}$ for configurations consisting of $11$ layers with perfectly conducting inclusions depicted in Figure~\ref{fig:layers_inclusions}, where the interfaces $\Gamma_\ell,0\leq \ell\leq 9$ are given by grating profiles $F_\ell(x_1)=-\ell H+2.5\varepsilon\cos{x_1}, H=3, 0\leq \ell\leq 9$ with $\varepsilon=0.1$, under normal incidence, and various values of wavenumbers $k_\ell$. The wavenumbers corresponding to layers with inclusions were selected to be Wood anomalies. Parameters $A=300$, $M=256$ were selected so that to lead to conservation of energy errors of the order $10^{-4}$. \label{comp8aW}}
\end{center}
   \end{table}   

\section{Conclusions}
   
We presented a sweeping preconditioner for the QO DD formulation of Helmholtz transmission problems in two dimensional periodic layered media.  Our QO DD formulation is built upon transmission operators who's construction relies on high-order shape deformation expansions of periodic layer DtN operators.  We used robust boundary integral equation formulations to represent the RtR operators, which were discretized via high-order Nystr\"om discretizations. The sweeping preconditioners are particularly effective in the case when the subdomain partitions consist of horizontal layers, at least when the boundaries of the layers do not contain cross points.   Extensions to cases when cross points are present, and to three dimensional cases are underway. We are also exploring strategies to parallelize the sweeping preconditioners.

\section*{Acknowledgments}
 Catalin Turc gratefully acknowledges support from NSF through contract DMS-1614270.

\bibliographystyle{abbrv}
\bibliography{biblioLayer}

\begin{thebibliography}{10}

\bibitem{arens2010scattering}
T.~Arens.
\newblock {\em Scattering by biperiodic layered media: The integral equation
  approach}.
\newblock Habilitation treatise, Karlsruhe Institute of Technology (KIT), 2010.

\bibitem{boubendirDDM}
Y.~Boubendir, X.~Antoine, and C.~Geuzaine.
\newblock A quasi-optimal non-overlapping domain decomposition algorithm for
  the {H}elmholtz equation.
\newblock {\em Journal of Computational Physics}, 231(2):262--280, 2012.

\bibitem{bruno2016superalgebraically}
O.~P. Bruno, S.~P. Shipman, C.~Turc, and S.~Venakides.
\newblock Superalgebraically convergent smoothly windowed lattice sums for
  doubly periodic {G}reen functions in three-dimensional space.
\newblock {\em Proc. R. Soc. A}, 472(2191):20160255, 2016.

\bibitem{bruno2017three}
O.~P. Bruno, S.~P. Shipman, C.~Turc, and S.~Venakides.
\newblock Three-dimensional quasi-periodic shifted {G}reen function throughout
  the spectrum--including {W}ood anomalies.
\newblock {\em Proc. R. Soc. A}, 473(2207):20170242, 2017.

\bibitem{Petit}
M.~Cadilhac and R.~Petit.
\newblock On the diffraction problem in electromagnetic theory: a discussion
  based on concepts of functional analysis including an example of practical
  application.
\newblock In {\em Huygens' principle 1690--1990: Theory and applications ({T}he
  {H}ague and {S}cheveningen, 1990)}, volume~3 of {\em Stud. Math. Phys.},
  pages 249--272. North-Holland, Amsterdam, 1992.

\bibitem{cho2015robust}
M.~H. Cho and A.~H. Barnett.
\newblock Robust fast direct integral equation solver for quasi-periodic
  scattering problems with a large number of layers.
\newblock {\em Optics Express}, 23(2):1775--1799, 2015.

\bibitem{Depres}
B.~Despr{\'e}s.
\newblock D\'ecomposition de domaine et probl\`eme de {H}elmholtz.
\newblock {\em C. R. Acad. Sci. Paris S\'er. I Math.}, 311(6):313--316, 1990.

\bibitem{dominguez2016well}
V.~Dominguez, M.~Lyon, C.~Turc, et~al.
\newblock Well-posed boundary integral equation formulations and nystr{\"o}m
  discretizations for the solution of helmholtz transmission problems in
  two-dimensional lipschitz domains.
\newblock {\em Journal of Integral Equations and Applications}, 28(3):395--440,
  2016.

\bibitem{engquist2011sweeping}
B.~Engquist and L.~Ying.
\newblock Sweeping preconditioner for the helmholtz equation: hierarchical
  matrix representation.
\newblock {\em Communications on pure and applied mathematics}, 64(5):697--735,
  2011.

\bibitem{Gander1}
M.~J. Gander, F.~Magoul{\`e}s, and F.~Nataf.
\newblock Optimized {S}chwarz methods without overlap for the {H}elmholtz
  equation.
\newblock {\em SIAM J. Sci. Comput.}, 24(1):38--60 (electronic), 2002.

\bibitem{gander2016class}
M.~J. Gander and H.~Zhang.
\newblock A class of iterative solvers for the helmholtz equation:
  Factorizations, sweeping preconditioners, source transfer, single layer
  potentials, polarized traces, and optimized schwarz methods.
\newblock {\em arXiv preprint arXiv:1610.02270}, 2016.

\bibitem{hong2017stable}
Y.~Hong and D.~P. Nicholls.
\newblock A stable high-order perturbation of surfaces method for numerical
  simulation of diffraction problems in triply layered media.
\newblock {\em Journal of Computational Physics}, 330:1043--1068, 2017.

\bibitem{jerez2017multitrace}
C.~Jerez-Hanckes, C.~P{\'e}rez-Arancibia, and C.~Turc.
\newblock Multitrace/singletrace formulations and domain decomposition methods
  for the solution of helmholtz transmission problems for bounded composite
  scatterers.
\newblock {\em Journal of Computational Physics}, 2017.

\bibitem{kusmaul}
R.~Kussmaul.
\newblock Ein numerisches {V}erfahren zur {L}\"osung des {N}eumannschen
  {A}ussenraumproblems f\"ur die {H}elmholtzsche {S}chwingungsgleichung.
\newblock {\em Computing (Arch. Elektron. Rechnen)}, 4:246--273, 1969.

\bibitem{LaiKobayashiBarnett2015}
J.~Lai, M.~Kobayashi, and A.~H. Barnett.
\newblock A fast and robust solver for the scattering from a layered periodic
  structure containing multi-particle inclusions.
\newblock {\em Journal of Computational Physics}, 298:194--208, 2015.

\bibitem{lecouvez2014quasi}
M.~Lecouvez, B.~Stupfel, P.~Joly, and F.~Collino.
\newblock Quasi-local transmission conditions for non-overlapping domain
  decomposition methods for the helmholtz equation.
\newblock {\em Comptes Rendus Physique}, 15(5):403--414, 2014.

\bibitem{martensen}
E.~Martensen.
\newblock \"{U}ber eine {M}ethode zum r\"aumlichen {N}eumannschen {P}roblem mit
  einer {A}nwendung f\"ur torusartige {B}erandungen.
\newblock {\em Acta Math.}, 109:75--135, 1963.

\bibitem{nicholls2012three}
D.~P. Nicholls.
\newblock Three-dimensional acoustic scattering by layered media: a novel
  surface formulation with operator expansions implementation.
\newblock {\em Proc. R. Soc. A}, 468(2139):731--758, 2012.

\bibitem{nicholls2018stable}
D.~P. Nicholls.
\newblock Stable, high-order computation of impedance--impedance operators for
  three-dimensional layered medium simulations.
\newblock {\em Proc. R. Soc. A}, 474(2212):20170704, 2018.

\bibitem{nicholls2004shape}
D.~P. Nicholls and F.~Reitich.
\newblock Shape deformations in rough-surface scattering: cancellations,
  conditioning, and convergence.
\newblock {\em JOSA A}, 21(4):590--605, 2004.

\bibitem{Delourme}
P.~B. O and B.~Delourme.
\newblock Rapidly convergent two-dimensional quasi-periodic {G}reen function
  throughout the spectrum‚ including {W}ood anomalies.
\newblock {\em Journal of Computational Physics}, 262(Supplement C):262 -- 290,
  2014.

\bibitem{perez2018domain}
C.~P{\'e}rez-Arancibia, S.~Shipman, C.~Turc, and S.~Venakides.
\newblock Domain decomposition for quasi-periodic scattering by layered media
  via robust boundary-integral equations at all frequencies.
\newblock {\em arXiv preprint arXiv:1801.09094}, 2018.

\bibitem{schadle2007domain}
A.~Sch{\"a}dle, L.~Zschiedrich, S.~Burger, R.~Klose, and F.~Schmidt.
\newblock Domain decomposition method for {M}axwell's equations: {S}cattering
  off periodic structures.
\newblock {\em Journal of Computational Physics}, 226(1):477--493, 2007.

\bibitem{vion2014double}
A.~Vion and C.~Geuzaine.
\newblock Double sweep preconditioner for optimized schwarz methods applied to
  the helmholtz problem.
\newblock {\em Journal of Computational Physics}, 266:171--190, 2014.

\bibitem{zepeda2016method}
L.~Zepeda-N{\'u}nez and L.~Demanet.
\newblock The method of polarized traces for the 2d helmholtz equation.
\newblock {\em Journal of Computational Physics}, 308:347--388, 2016.

\end{thebibliography}

\end{document}